\documentclass[12pt,a4paper,reqno]{amsart}
\usepackage[utf8]{inputenc}
\usepackage[english]{babel}
\usepackage{amsmath}
\usepackage{enumerate}
\usepackage{mathtools}
\usepackage{mathrsfs}
\usepackage{soul, xcolor}
\usepackage{wasysym}
\usepackage{amsmath,amssymb,xpatch,relsize,graphics,graphicx,float, amsthm}
\usepackage{array}  
\usepackage{lscape} 
\usepackage{booktabs} 
\usepackage{caption} 
\usepackage{float,graphicx}
\usepackage{subcaption}
\usepackage{epstopdf} 
\xapptocmd\normalsize{%
	\abovedisplayskip=12pt plus 3pt minus 9pt
	\abovedisplayshortskip=0pt plus 3pt
	\belowdisplayskip=12pt plus 3pt minus 9pt
	\belowdisplayshortskip=7pt plus 3pt minus 4pt
}{}{}
\theoremstyle{definition}
\newtheorem{definition}{Definition}[section]

\theoremstyle{plain}
\newtheorem{theorem}[definition]{Theorem}
\newtheorem{corollary}[definition]{Corollary}
\newtheorem{lemma}[definition]{Lemma}

\numberwithin{equation}{section}

\title[The sharp refined Bohr inequalities]{The sharp refined Bohr inequalities  for a subclass of close-to-convex harmonic mappings}
\author[A. Kumar]{Ayush Kumar}
\address{Department of Mathematics, University of Delhi, Delhi--110 007, India}
\email{akumar7@maths.du.ac.in}



\allowdisplaybreaks

\begin{document}
\maketitle
\begin{abstract}
Let $\mathcal{H}$ be the class of normalized complex valued harmonic functions  $ f = h + \overline{g}$ defined on the unit disk $\mathbb{D}$, where $h$ and $g$ are analytic functions with the normalization conditions $h(0) = h'(0) - 1 = 0$ and $g(0) = 0$.  
For the class $R_H^{0}(\gamma, \delta, \lambda)$ ( $0 \leq \lambda < \gamma \leq \delta$)   consisting of functions \( f = h+\bar{g} \in \mathcal{H}\) satisfying the condition $f_{\overline{z}}(0)=0$ and the inequality
$ Re(\gamma h'(z)+\delta z h''(z) +(\frac{\delta - \gamma}{2})z^2 h'''(z)-\lambda)> |\gamma g'(z)+\delta z g''(z) +(\frac{\delta - \gamma}{2})z^2 g'''(z)|$, we obtain sharp improved Bohr Phenomenon, refined Bohr radius and the Bohr-Rogosinski inequality for the class $R_H^{0}(\gamma, \delta, \lambda)$.
\end{abstract}
\section{Introduction}
Let $\mathfrak{A}$ denote the class of analytic functions $f$ in the unit disk $\mathbb{D}$ which satisfies $|f(z)|<1$ for all $z \in \mathbb{D}$. If $f\in \mathfrak{A}$ is of the form $f(z)=\sum_{m=0}^{\infty}a_m z^m$, then $M_f(r)=\sum_{m=0}^{\infty}|a_m| |z|^m$ is the majorant series of $f$. In 1914, Bohr \cite{boh1914} proved that the majorant series of the function $f$ satisfies the inequality
\begin{equation}\label{eq1}
\sum_{m=0}^{\infty}|a_m| |z|^m \leq 1
\end{equation}
for all $|z|=r\le \frac{1}{6}$. The inequality in equation \eqref{eq1} is known as the Bohr inequality. The Bohr inequality was later carried forward by Wiener, Riesz and Schur independently  for $|z|\le \frac{1}{3}$, and showed that $\frac{1}{3}$ is the best possible. This radius $\frac{1}{3}$ is known as the Bohr radius for class $\mathfrak{A}$.

The Bohr inequality has  emerged as an active area of research for operator algebraists after Dixon \cite{Dix1995} used it to settle in  negative a conjecture that a Banach algebra satisfying a nonunital Von Neumann inequality is necessarily an operator algebra. Further, Paulson and Singh \cite{Pau2004} (also see, \cite{Pau2002, Pau2006}) and Blasco \cite{Bla2010} extended the Bohr inequality in connection with Banach algebra. Boas and Khavinson \cite{Boa1997}, and Aizenberg \cite{Aiz2000, Aiz2005, Aiz2001}  obtained the inequality for several complex variables.  The Bohr inequality has attracted many  researcher's attention in the geometric function theory. The study of Bohr phenomenon for different classes of functions satisfying various conditions becomes a subject of great interest during the past several years which yield an extensive research done by many authors \cite{Ahm2023, Ali2019,  Ana2021, Ahu2020, Jai2020, Liu2023, Kay2022, Kum2021, Kha2022, Gan2022}. In 2010, Abu-Muhanna at el.  \cite{Abu2010} initiated the study of  Bohr phenomenon for the class of harmonic mapping. Later, Abu-Muhanna \cite{Abu2014} investigated the phenomenon for harmonic functions and subsequently by a number of authors \cite{Liu2023, All2021, Hua2021, Kay2018J, Liu2019, AHa2024}.

 A harmonic mapping within the unit disk \( \mathbb{D} \) is defined as a complex function \( f = u + iv \) in \( \mathbb{D} \) which  satisfies the Laplace equation \( \Delta f = 4f_{z\bar{z}} = 0 \), with \( f_{z} = \dfrac{f_{x} - if_{y}}{2} \), \( f_{\bar{z}} = \dfrac{f_{x} + if_{y}}{2} \), where \( u \) and \( v \) representing real-valued harmonic functions in \( \mathbb{D} \). Consequently, the function \( f \) can be expressed in its canonical form as \( f = h + \overline{g} \), where both \( h \) and \( g \) are analytic functions in \( \mathbb{D} \). The Jacobian determinant \( J_f \) of the mapping \( f = h + \overline{g} \) is calculated as \( J_f = |h'|^2 - |g'|^2 \). The function \( f \) is sense-preserving in \( \mathbb{D} \) if \( J_f(z) > 0 \) for all \( z \in \mathbb{D} \). Therefore, \( f \) is locally univalent and sense-preserving in \( \mathbb{D} \) if and only if \( J_f > 0 \) in \( \mathbb{D} \). Alternatively, if \( h' \neq 0 \) in \( \mathbb{D} \), the dilation \( \omega_f := \dfrac{g'}{h'} \) must satisfy the condition \( |\omega(z)| < 1 \) for all \( z \in \mathbb{D} \). Detailed information about the harmonic mapping can be found in \cite{Clu1984, Dur2004}. The harmonic extension of the classical Bohr theorem was first established by  Kayumov et al. \cite{All2022};  since then, investigating  the Bohr-type inequalities for certain class of harmonic mappings has become an interesting topic of
research in geometric function theory.

Let \( \mathcal{H} \) represent the set of all complex-valued harmonic functions \( f = h + \overline{g} \) defined on the unit disk \( \mathbb{D} \), where \( h \) and \( g \) are analytic functions with the normalization conditions \( h(0) = h'(0) - 1 = 0 \) and \( g(0) = 0 \). We denote by \( \mathcal{H}_0 \) the subset of \( \mathcal{H} \)  defined by \( \mathcal{H}_0 = \{ f = h + \overline{g} \in \mathcal{H} : g'(0) = 0 \} \). Consequently, any function \( f = h + \overline{g} \in \mathcal{H}_0 \) can be expressed as

\begin{equation}\label{eq2}
h(z) = z + \sum_{m=2}^{\infty} a_m z^m \quad \text{and} \quad g(z) = \sum_{m=2}^{\infty} b_m z^m.
\end{equation}

 In 2010, Abu-Muhana \cite{Abu2010} defined the Bohr phenomenon, for the class of all analytic  functions $g$ subordinate to a fixed analytic function $f$ (that is $\mathcal{S}(f)$),  using the Euclidean distance as follows:
 \begin{definition}(\cite{Abu2010}) For any analytic functions $f=\sum_{n=0}^\infty a_nz^n$ and  $g=\sum_{n=0}^\infty b_nz^n\in\mathcal{S}(f)$, the class $\mathcal{S}(f)$ is said to satisfy the Bohr phenomenon if there is a $r^*$, $0<r^*\leq1$, such that the inequality $\sum_{n=0}^\infty |b_nz^n| \leq d(f(0), f(\mathbb{D}))$, holds for $|z|<r^*$.
  \end{definition}
 Here, $d(f(0), f(\mathbb{D}))$ denote the Euclidean distance between $f(0)$ and the boundary of the image of the unit disk under the mapping $f$. In particular, for $f(\mathbb{D})=\mathbb{D}$, $d(f(0), \partial f(\mathbb{D}))=1-|f(0)|$ and the inequality $\sum_{n=0}^\infty |a_n|\leq d(f(0),\partial f(\mathbb{D}))$ reduces to $\sum_{n=0}^\infty |a_nz^n|\leq1$. Thus the equation \eqref{eq1} can be rewritten as
 \begin{equation*}
 d\left(\sum_{n=0}^\infty |a_n|,|a_0|\right)=\sum_{n=1}^\infty |a_n|\leq1-|f(0)|=d(f(0), \partial f(\mathbb{D})),
 \end{equation*}
where $d$ is the Euclidean distance. Kayumov et al. \cite{Kay2018M} defined the similar definition for harmonic functions.

\begin{definition}(\cite{Kay2018M})
Let $f\in\mathcal{H}_0$ be given be \eqref{eq2}. Then the Bohr phenomenon is to find
a constant $0<\rho^*\leq1$ such that the inequality $|z|+\sum_{n=2}^\infty (|a_n|+|b_n|)|z|^n\leq d(f(0), \partial f(\mathbb{D}))$ holds for $|z|=r\leq\rho^*$. The largest such radius $\rho^*$ is called the Bohr radius for the class $\mathcal{H}_0$.
\end{definition}

In a similar vein to the Bohr radius, the Bohr–Rogosinski radius has been established \cite{Rog1923} and is defined as follows: If \( f \in \mathfrak{A} \), then for \( N \geq 1 \), we have \( |S_N(z)| < 1 \) within the disk \( \mathbb{D}_{1/2} \), and this radius is referred to as the Bohr–Rogosinski sum \( R^f_N(z) \) of the function \( f \), given by

\[
R^f_N(z) := |f(z)| + \sum_{m=N}^{\infty} |a_m| r^m, \quad |z| = r.
\]

It is noteworthy that for \( N = 1 \), the expression simplifies to the classical Bohr sum, where \( f(0) \) is replaced by \( |f(z)| \). In 2017, Kayumov and Ponnusamy \cite{kay2017} demonstrated the following significant result regarding the Bohr–Rogosinski radius for analytic functions.

\begin{theorem}\cite{kay2017}\label{nth1}
Suppose that \( f(z) = \sum_{m=0}^{\infty} a_m z^m \) is analytic in the unit disk \( \mathbb{D} \) and \( |f(z)| < 1 \) in \( \mathbb{D} \). Then

\begin{equation*}
|f(z)| + \sum_{m=N}^{\infty} |a_m| r^m \leq 1, \quad \text{for } r \leq R_N,
\end{equation*}
where \( R_N \) is the positive root of the equation \( 2(1+r)r^N - (1-r^2) = 0 \). The radius \( R_N \) is the best possible. Moreover,

\begin{equation*}
|f(z)|^2 + \sum_{m=N}^{\infty} |a_m| r^m \leq 1, \quad \text{for } r \leq R_N',
\end{equation*}
where \( R'_{N} \) is the positive root of the equation \( (1+r)r^N - (1-r^2) = 0 \). The radius \( R'_{N} \) is the best possible.
\end{theorem}

For more results on Bohr-Rogosinski inequality, we refer to  \cite{Aiz2012, Aiz2005S, Aiz2009, Alk2020, Kay2021, Pon2020} and references therein.
Let $\lfloor x\rfloor$ denotes the largest integer no more than $x$, where $x$ is the real number. A refined version of Bohr-Rogosinski inequality was proved  by Liu et al. \cite{LIU2021}

\begin{theorem}\label{nth2}\cite{LIU2021}
Suppose that $f\in\mathfrak{A}$ and $f(z)=\sum_{n=1}^{\infty}a_nz^n$. For $N\in\mathbb{N}$, let $t=\lfloor(N-1)/2\rfloor$. Then
\begin{equation*}
\begin{aligned}
|f(z)| + \sum_{m=N}^{\infty}|a_m|r^m &+  \text{sgn}(t) \sum_{m=1}^{t} |a_m|^2 \frac{r^N}{1 - r}\\
&+  \left( \frac{1}{1+a_0} + \frac{r}{1 - r} \right) \sum_{m=t+1}^{\infty} (|a_m| + |b_m|)^2 r^{2m}\leq1
\end{aligned}
\end{equation*}
for $|z|=r\leq R_N$, where $R_N$ is an in Theorem \ref{nth1}. The radius $R_N$ is sharp.
Also,
\begin{equation*}
\begin{aligned}
|f(z)|^2 + \sum_{m=N}^{\infty}|a_m|r^m &+  \text{sgn}(t) \sum_{m=1}^{t} |a_m|^2 \frac{r^N}{1 - r}\\
&+  \left( \frac{1}{1+a_0} + \frac{r}{1 - r} \right) \sum_{m=t+1}^{\infty} (|a_m| + |b_m|)^2 r^{2m}\leq1
\end{aligned}
\end{equation*}
for $|z|=r\leq R'_N$, where $R_N'$ is an in Theorem \ref{nth1}. The radius $R'_N$ is sharp.
\end{theorem}

Recently, Ahamed \cite{Aha2024M} proved the inequalities $S^f_{\mu, \beta, n, N}(r)\leq d(f(0), \partial f(\mathbb{D}))$, where
\begin{equation}
\begin{aligned}\label{eq13a}
S^f_{\mu, \beta, n, N}(r) := |f(z)|^n + \sum_{m=N}^{\infty} (|a_m| + |b_m|) r^m &+ \mu \, \text{sgn}(t) \sum_{m=1}^{t} (|a_m| + |b_m|)^2 \frac{r^N}{1 - r}\\
&+  \beta \left( 1 + \frac{r}{1 - r} \right) \sum_{m=t+1}^{\infty} (|a_m| + |b_m|)^2 r^{2m},
\end{aligned}
\end{equation}
 to establish  harmonic analog of Theorem \ref{nth1} and Theorem \ref{nth2} for certain harmonic classes.

Chichra \cite{Chi1976} first introduced the class \( W(\alpha) \) in 1977, which includes normalized analytic functions \( h \) that satisfy the condition \( \text{Re}(h'(z) + \alpha z h''(z)) > 0 \) for \( z \in \mathbb{D} \) and \( \alpha \geq 0 \). Furthermore, Chichra \cite{Chi1976} demonstrated that functions within the \( W(\alpha) \) class form a subset of functions that are close to being convex in \( \mathbb{D} \). In 2014, Nagpal and Ravichandran \cite{Nag2014} explored the class

\[
W_{\mathcal{H}}^0 = \{ f = h + \overline{g} \in \mathcal{H} : \text{Re}(h'(z) + z h''(z)) > |g'(z) + z g''(z)| \text{ for } z \in \mathbb{D} \}
\]
and derived coefficient bounds for functions in the class \( W^0_{\mathcal{H}} \). In 2019, Ghosh and Vasudevarao examined the class \( W^0_{\mathcal{H}}(\alpha) \), where

\[
W_{\mathcal{H}}^0(\alpha) = \{ f = h + \overline{g} \in \mathcal{H} : \text{Re}(h'(z) + \alpha z h''(z)) > |g'(z) + \alpha z g''(z)| \text{ for } z \in \mathbb{D} \}.
\]
Later, some  authors \cite{Raj2021, Yas2021} generalized the class $W_{\mathcal{H}}^0(\alpha)$.
In 2021, \c Cakmak et al. \cite{Cak2022} introduced a new class $R_H^{0}(\gamma, \delta, \lambda)$ for $0 \leq \lambda < \gamma \leq \delta$ of complex-valued harmonic functions $f\in \mathcal{H}_0$ satisfying the inequality
$$ Re\left(\gamma h'(z)+\delta z h''(z) +\left(\frac{\delta - \gamma}{2}\right)z^2 h'''(z)-\lambda\right)> \left|\gamma g'(z)+\delta z g''(z) +\left(\frac{\delta - \gamma}{2}\right)z^2 g'''(z)\right| $$
 It was proved that the members of the class are close-to-convex. The authors obtained coefficient bounds, growth estimates, and sufficient  condition for this class. Further, they investigated the radii of fully starlikeness and fully convexity of the class.
In 2022 Wang et al. \cite{wang2022d} investigated the class
\[\mathcal{W}_{\mathscr{H}}^0 (\alpha, \beta, \gamma) = \left\{ f = h + \overline{g} \in \mathscr{H}^0 : \Re \left( h'(z) + \alpha z h''(z) + \gamma z^2 h'''(z) - \beta \right) \right.\]
\[\left. \phantom{\mathcal{W}_{\mathscr{H}}^0 (\alpha, \beta, \gamma) = } > |g'(z) + \alpha z g''(z) + \gamma z^2 g'''(z)| \ (z \in \mathbb{D}) \right\}.\]
Authors obtained different results associated with this class related to Bohr Phenomenon.

The primary objective of this paper is to explore various Bohr inequalities related to harmonic functions comprehensively. Motivated by the  work done in the papers \cite{Aha2022, Aha2024M}, in the present paper we determine  improved Bohr phenomenon, Bohr–Rogosinski radius and refined Bohr radius, for the class $R_H^{0}(\gamma, \delta, \lambda)$  in Section  \ref{sec3}, \ref{sec4}  and \ref{sec2} respectively.

\section{Preliminary results}

 We need the following Lemmas  proved by \c Cakmak et al. \cite{Cak2022} to prove our  main results:

\begin{lemma}\label{lm1}\cite{Cak2022}
Let $f \in R_H^{0}(\gamma, \delta, \lambda)$. Then for any $m \geq 2$:
\begin{enumerate}
    \item $|a_m| + |b_m| \leq \dfrac{4(\gamma-\lambda)}{m^2[2\gamma+(\delta - \gamma)(m-1)]};$
    \item $||a_m| - |b_m|| \leq \dfrac{4(\gamma-\lambda)}{m^2[2\gamma+(\delta - \gamma)(m-1)]};$
    \item $|a_m| \leq \dfrac{4(\gamma-\lambda)}{m^2[2\gamma+(\delta - \gamma)(m-1)]}.$
\end{enumerate}
All these inequalities are sharp and equality holds for the function $f$ given by
\begin{equation}\label{eq3}
    f(z) = z + \sum_{m=2}^{\infty} \frac{2(\gamma-\lambda)z^m}{m^2[2\gamma+(\delta - \gamma)(m-1)]}.
\end{equation}
\end{lemma}

\begin{lemma}\label{lm2}\cite{Cak2022}
Let $f \in R_H^{0}(\gamma, \delta, \lambda)$. Then
\begin{align}\label{eq4}
    |z| + 4(\gamma-\lambda)\sum_{m=2}^{\infty} \frac{(-1)^{m-1}|z|^m}{m^2[2\gamma+(\delta - \gamma)(m-1)]}& \leq |f(z)|\nonumber \\
     \leq |z| + 4(\gamma-\lambda)\sum_{m=2}^{\infty} \frac{|z|^m}{m^2[2\gamma+(\delta - \gamma)(m-1)]}.
\end{align}
Both the inequalities are sharp for the function $f$ given by
$$f(z) = z + \sum_{m=2}^{\infty} \frac{2(\gamma-\lambda)\bar{z}^m}{m^2[2\gamma+(\delta - \gamma)(m-1)]}.$$
\end{lemma}

\begin{lemma}\label{lm3}\cite{Cak2022}
Let $f \in R_H^{0}(\gamma, \delta, \lambda)$. Then for any $m \geq 2$;
\begin{equation}\label{eq5}
    |b_m| \leq \frac{2(\gamma-\lambda)}{m^2[2\gamma+(\delta - \gamma)(m-1)]}.
\end{equation}
The inequality is sharp for the function $f$ given by
\begin{equation}\label{eq6}
    f(z) = z + \sum_{m=2}^{\infty} \frac{2(\gamma-\lambda)\bar{z}^m}{m^2[2\gamma+(\delta - \gamma)(m-1)]}.
\end{equation}
\end{lemma}

Before presenting the primary findings of this paper, we first outline the definition of the dilogarithm. The dilogarithm \( \text{Li}_2(z) \) is expressed through the power series:

\[
\text{Li}_2(z) = \sum_{m=1}^{\infty} \frac{z^m}{m^2}, \quad |z| < 1.
\]
This definition and its nomenclature arise from the similarity to the Taylor series expansion of the ordinary logarithm around the point 1:

\[
-\log(1 - z) = \sum_{m=1}^{\infty} \frac{z^m}{m}, \quad |z| < 1,
\]
which similarly leads to the definition of the polylogarithm:

\[
Li_{n}(z) = \sum_{m=1}^{\infty} \frac{z^m}{m^n}, \quad |z| < 1, \quad n = 1, 2, 3, \ldots
\]
Additionally, the relationship

\[
\frac{d}{dz} \left( {Li_n}(z) \right) = \frac{1}{z} {Li}_{n-1}(z), \quad n \geq 2,
\]
is noteworthy. This relationship is crucial for understanding the behavior of the polylogarithm. The extension of the domain for \( {Li_n}(z) \) is evident and can be shown through induction, allowing us to expand its definition to the cut plane \( \mathbb{C} \setminus [1, \infty) \). Specifically, the analytic continuation of the dilogarithm is given by:

\[
\text{Li}_2(z) = -\int_0^z \frac{\log(1 - u)}{u} \, du, \quad z \in \mathbb{C} \setminus [1, \infty).
\]
\section{Improved Bohr Phenomenon}\label{sec3}

In the present section, we prove various improved Bohr Phenomenon for the class $R_H^{0}(\gamma, \delta, \lambda)$.

\begin{theorem}\label{th2}
  Let $f \in R_H^{0}(\gamma, \delta, \lambda)$ ($0 \leq \lambda < \gamma \leq \delta$). Then for any $p \geq 1$,
  $$|z| + \sum_{m=2}^{\infty}(|a_m|+|b_m|)|z|^m + \sum_{m=2}^{\infty}(|a_m|+|b_m|)^p|z|^{pm} \leq d(f(0),\partial f(\mathbb{D}))$$
  for $|z|=r \leq r_p(\gamma, \delta, \lambda)$, where $r_p(\gamma, \delta, \lambda)$ is the unique root in $(0,1)$ of
  \begin{equation*}
  \begin{split}
    r + 4(\gamma-\lambda)\sum_{m=2}^{\infty} \frac{r^m}{m^2[2\gamma+(\delta - \gamma)(m-1)]} &+ \sum_{m=2}^{\infty}\left [4(\gamma-\lambda) \frac{r^m}{m^2[2\gamma+(\delta - \gamma)(m-1)]} \right]^p\\
    &= 1 + 4(\gamma-\lambda)\sum_{m=2}^{\infty} \frac{(-1)^{m-1}}{m^2[2\gamma+(\delta - \gamma)(m-1)]}.
  \end{split}
  \end{equation*}
  The radius $r_p(\gamma, \delta, \lambda)$ is the best possible.
\end{theorem}
\begin{proof}
  Let $f \in R_H^{0}(\gamma, \delta, \lambda)$. Then by Lemma \ref{lm2}, for $|z|=r$, we get
  $$r + 4(\gamma-\lambda)\sum_{m=2}^{\infty} \frac{(-1)^{m-1}r^m}{m^2[2\gamma+(\delta - \gamma)(m-1)]} \leq |f(z)| \leq r + 4(\gamma-\lambda)\sum_{m=2}^{\infty} \frac{r^m}{m^2[2\gamma+(\delta - \gamma)(m-1)]}.$$
  Since $f(0)=0$, $|f(z)|=|f(z)-f(0)|$. Therefore, \\
  \begin{equation*}
    \begin{split}
      \liminf_{r\rightarrow 1^{-}}\left(r + \sum_{m=2}^{\infty} \frac{(-1)^{m-1}4(\gamma-\lambda)r^m}{m^2[2\gamma+(\delta - \gamma)(m-1)]}\right) &\leq \liminf_{r\rightarrow 1^{-}}|f(z)-f(0)|\\
      &\leq \liminf_{r\rightarrow 1^{-}}\left(r + \sum_{m=2}^{\infty} \frac{4(\gamma-\lambda)r^m}{m^2[2\gamma+(\delta - \gamma)(m-1)]}\right).
    \end{split}
  \end{equation*}
  For $0 \leq \lambda < \gamma \leq \delta$, it is easy to see that the radius of convergence of real power series given by
  \begin{equation}\label{eq14}
    r + \sum_{m=2}^{\infty} \dfrac{(-1)^{m-1}4(\gamma-\lambda)r^m}{m^2[2\gamma+(\delta - \gamma)(m-1)]}
  \end{equation}
  is 1 which means that interval of convergence is (-1,1).\\
  The power series \eqref{eq14} is convergent for $r=1$ as well. So by using Abel's Theorem for power series we conclude that it is uniformly convergent in [0,1], that means the function $h(r)$ is continuous on [0,1]  given by
  \begin{equation}\label{eq15}
    h(r):= r + \sum_{m=2}^{\infty} \dfrac{(-1)^{m-1}4(\gamma-\lambda)r^m}{m^2[2\gamma+(\delta - \gamma)(m-1)]}
  \end{equation}
  for $r \in (-1,1]$. Particularly, $h(r)$ in \eqref{eq15} is left continuous at $r=1$. That is,
  \begin{align*}
  \lim_{r\rightarrow 1^{-}}\left(r + \sum_{m=2}^{\infty} \dfrac{(-1)^{m-1}4(\gamma-\lambda)r^m}{m^2[2\gamma+(\delta - \gamma)(m-1)]}\right)&=\lim_{r\rightarrow 1^-}h(r)=h(1)\\
  &=1 + \sum_{m=2}^{\infty} \dfrac{(-1)^{m-1}4(\gamma-\lambda)}{m^2[2\gamma+(\delta - \gamma)(m-1)]}\\
  &=1 + \sum_{m=2}^{\infty} \lim_{r\rightarrow 1^{-}}\dfrac{(-1)^{m-1}4(\gamma-\lambda)r^m}{m^2[2\gamma+(\delta - \gamma)(m-1)]}.
  \end{align*}

  Now, it can be easily seen that
  \begin{equation*}
  \begin{aligned}
      d(f(0),\partial f(\mathbb{D})) &= \liminf_{|z|=r\rightarrow 1^{-}}|f(z)-f(0)|\\
       & \geq  \liminf_{r\rightarrow 1^{-}}\left(r + \sum_{m=2}^{\infty} \dfrac{(-1)^{m-1}4(\gamma-\lambda)r^m}{m^2[2\gamma+(\delta - \gamma)(m-1)]}\right)\\
       &= \liminf_{r\rightarrow 1^{-}}\left(r + \sum_{m=2}^{\infty} \dfrac{(-1)^{m-1}4(\gamma-\lambda)r^m}{m^2[2\gamma+(\delta - \gamma)(m-1)]}\right) \\
       &= 1 + \sum_{m=2}^{\infty} \dfrac{(-1)^{m-1}4(\gamma-\lambda)}{m^2[2\gamma+(\delta - \gamma)(m-1)]}.
  \end{aligned}
  \end{equation*}

  Thus,
  \begin{equation}\label{eq16}
  d(f(0),\partial f(\mathbb{D})) \geq 1 + \sum_{m=2}^{\infty} \dfrac{(-1)^{m-1}4(\gamma-\lambda)}{m^2[2\gamma+(\delta - \gamma)(m-1)]}.
  \end{equation}
  In view of Lemma \ref{lm1}, for $|z|=r$, we have
  \begin{equation*}
  \begin{split}
    |z|+\sum_{m=2}^{\infty}(|a_m+b_m|)|z|^m+\sum_{m=2}^{\infty}(|a_m+b_m|)^p|z|^{pm} & \leq |z|+\sum_{m=2}^{\infty}\dfrac{4(\gamma-\lambda)|z|^{m}}{(m^2[2\gamma+(\delta - \gamma)(m-1)])}\\
    &+\sum_{m=2}^{\infty}\dfrac{4^p(\gamma-\lambda)^p|z|^{pm}}{(m^2[2\gamma+(\delta - \gamma)(m-1)])^p}\\
    &= r+\sum_{m=2}^{\infty}\dfrac{4(\gamma-\lambda) r^{m}}{(m^2[2\gamma+(\delta - \gamma)(m-1)])}\\
    &+ \sum_{m=2}^{\infty}\dfrac{4^p(\gamma-\lambda)^p r^{pm}}{(m^2[2\gamma+(\delta - \gamma)(m-1)])^p}.
  \end{split}
  \end{equation*}
  An easy calculation shows that
  \begin{align*}
  r+ \sum_{m=2}^{\infty}\dfrac{4(\gamma-\lambda) r^{m}}{(m^2[2\gamma+(\delta - \gamma)(m-1)])}+&\sum_{m=2}^{\infty}\dfrac{4^p(\gamma-\lambda)^p r^{pm}}{(m^2[2\gamma+(\delta - \gamma)(m-1)])^p}\\
   &\leq 1+\sum_{m=2}^{\infty}\dfrac{4(\gamma-\lambda) (-1)^{m-1}}{(m^2[2\gamma+(\delta - \gamma)(m-1)])}
  \end{align*}
  for $r\leq r_p(\gamma, \delta, \lambda)$, where $r_p(\gamma, \delta, \lambda)$ is a root of $k_1(r)=0$ in $(0,1)$ and $k_1:[0,1]\rightarrow \mathbb{R}$ is defined by
  \begin{equation}\label{eq18}
  \begin{split}
    k_1(r)=r+ \sum_{m=2}^{\infty}\dfrac{4(\gamma-\lambda) r^{m}}{(m^2[2\gamma+(\delta - \gamma)(m-1)])}&+\sum_{m=2}^{\infty}\dfrac{4^p(\gamma-\lambda)^p r^{pm}}{(m^2[2\gamma+(\delta - \gamma)(m-1)])^p}\\
     &- 1-\sum_{m=2}^{\infty}\dfrac{4(\gamma-\lambda) (-1)^{m-1}}{(m^2[2\gamma+(\delta - \gamma)(m-1)])}.
  \end{split}
  \end{equation}
  Clearly, the function $k_1(r)$ in \eqref{eq18} is a real-valued function, which is continuous on $[0,1]$ and differentiable on $(0,1)$. Since,
  \begin{equation}\label{eq19}
  \left|\sum_{m=2}^{\infty}\dfrac{4(\gamma-\lambda) (-1)^{m-1}}{(m^2[2\gamma+(\delta - \gamma)(m-1)])}\right|< 1.
  \end{equation}
  It can be easily seen using \eqref{eq19} that
  $$k_1(0):= -1-\sum_{m=2}^{\infty}\dfrac{4(\gamma-\lambda) (-1)^{m-1}}{(m^2[2\gamma+(\delta - \gamma)(m-1)])}<0.$$
  Also, it can be observed that
  $$\sum_{m=2}^{\infty}\dfrac{4(\gamma-\lambda)}{(m^2[2\gamma+(\delta - \gamma)(m-1)])} \geq \sum_{m=2}^{\infty}\dfrac{4(\gamma-\lambda)(-1)^{m-1}}{(m^2[2\gamma+(\delta - \gamma)(m-1)])}.$$
  A simple computation gives that
  \begin{eqnarray*}
  k_1(1) &=& \sum_{m=2}^{\infty}\dfrac{4(\gamma-\lambda)}{(m^2[2\gamma+(\delta - \gamma)(m-1)])}+\sum_{m=2}^{\infty}\dfrac{4^p(\gamma-\lambda)^p}{(m^2[2\gamma+(\delta - \gamma)(m-1)])^p}\\ &-&\sum_{m=2}^{\infty}\dfrac{4(\gamma-\lambda) (-1)^{m-1}}{(m^2[2\gamma+(\delta - \gamma)(m-1)])}\\
  &\geq& 4^p(\gamma-\lambda)^p\sum_{m=2}^{\infty}\dfrac{1}{(m^2[2\gamma+(\delta - \gamma)(m-1)])^p}.
  \end{eqnarray*}
  Therefore, $k_1(0)k_1(1)<0$ and hence using the Intermediate Value Theorem, it shows that $k_1(r)$ has roots in $(0,1)$. Next to show the uniqueness of root in $(0,1)$, it will be enough to show that $k_1$ is an strictly monotonic function in the interval $(0,1)$.\\
  A quick calculation shows that
  \[
  k'(r):=1+\sum_{m=2}^{\infty}\dfrac{4m(\gamma-\lambda) r^{m-1}}{(m^2[2\gamma+(\delta - \gamma)(m-1)])}+\sum_{m=2}^{\infty}\dfrac{4^p mp(\gamma-\lambda)^p r^{pm-1}}{(m^2[2\gamma+(\delta - \gamma)(m-1)])^p}>0
  \]
  for all $r \in (0,1)$, which shows that $k_1$ is strictly increasing function. Therefore, $k_1(r)$ has the unique root in $(0,1)$, say $r_p(\gamma,\delta,\lambda)$.\\
  Therefore, we get that $k_1(r_p(\gamma,\delta,\lambda))=0$ and hence from \eqref{eq18}, we obtain
  \begin{equation}\label{eq20}
  \begin{split}
    r_p(\gamma,\delta,\lambda)+ \sum_{m=2}^{\infty}\dfrac{4(\gamma-\lambda) r_p(\gamma,\delta,\lambda)^{m}}{(m^2[2\gamma+(\delta - \gamma)(m-1)])}&+\sum_{m=2}^{\infty}\dfrac{4^p(\gamma-\lambda)^p r_p(\gamma,\delta,\lambda)^{pm}}{(m^2[2\gamma+(\delta - \gamma)(m-1)])^p}\\ &= 1+\sum_{m=2}^{\infty}\dfrac{4(\gamma-\lambda) (-1)^{m-1}}{(m^2[2\gamma+(\delta - \gamma)(m-1)])}.
  \end{split}
  \end{equation}
  Next, we need to  show that $r_p(\gamma,\delta,\lambda)$ is the best possible, so consider the function
  $$f_1(z):=z + \sum_{m=2}^{\infty} \dfrac{(-1)^{m-1}4(\gamma-\lambda)z^m}{m^2[2\gamma+(\delta - \gamma)(m-1)]}.$$
  It is easy to see that $f_1\in R_H^0(\gamma,\delta,\lambda)$ and $f_1(0)=0$. Now at $z=-r$, it is easy to verify that
  \begin{align}\label{eq21}
    |f_1(-r)-f_1(0)|&=\left|-r+\sum_{m=2}^{\infty}\dfrac{4(\gamma-\lambda)(-r)^m}{(m^2[2\gamma+(\delta - \gamma)(m-1)])}\right|\nonumber\\
    &=r+\sum_{m=2}^{\infty}\dfrac{4(\gamma-\lambda)(-1)^{m-1}(r)^m}{(m^2[2\gamma+(\delta - \gamma)(m-1)])}.
  \end{align}
  Hence, in view of \eqref{eq21}, the distance $d$ between $f_1(0)$ and boundary of $f_1(\mathbb{D})$ is given by
  \begin{equation}\label{eq22}
    d(f_1(0),\partial f_1(\mathbb{D}))=\liminf_{r\rightarrow1^-}|f_1(-r)-f_1(0)|=1+\sum_{m=2}^{\infty}\dfrac{4(\gamma-\lambda)(-1)^{m-1}}{(m^2[2\gamma+(\delta - \gamma)(m-1)])}.
  \end{equation}
A simple and easy calculation using \eqref{eq20} and \eqref{eq22} for the function $f=f_1$ and $|z|=r=r_p(\gamma,\delta,\lambda)$, shows that
  \begin{align*}
    |z|+&\sum_{m=2}^{\infty}(|a_m+b_m|)|z|^m+\sum_{m=2}^{\infty}(|a_m+b_m|)^p|z|^{pm}\\
    &= r+\sum_{m=2}^{\infty}\dfrac{4(\gamma-\lambda) r^{m}}{(m^2[2\gamma+(\delta - \gamma)(m-1)])}+ \sum_{m=2}^{\infty}\dfrac{4^p(\gamma-\lambda)^p r^{pm}}{(m^2[2\gamma+(\delta - \gamma)(m-1)])^{p}}\\
    &= r_p(\gamma,\delta,\lambda)+\sum_{m=2}^{\infty}\dfrac{4(\gamma-\lambda) r_p(\gamma,\delta,\lambda)^{m}}{(m^2[2\gamma+(\delta - \gamma)(m-1)])}+ \sum_{m=2}^{\infty}\dfrac{4^p(\gamma-\lambda)^p r_p(\gamma,\delta,\lambda)^{pm}}{(m^2[2\gamma+(\delta - \gamma)(m-1)])^{p}}\\
   &=1+\sum_{m=2}^{\infty}\dfrac{4(\gamma-\lambda)(-1)^{m-1}}{(m^2[2\gamma+(\delta - \gamma)(m-1)])}=d(f_1(0),\partial f_1(\mathbb{D})).
  \end{align*}
  Therefore, $r_p(\gamma,\delta,\lambda)$ is the best possible.
\end{proof}

For a particular choice  of $\lambda, \gamma$ and $\delta$, we have the following corollary to Theorem \ref{th2}.
\begin{corollary}\label{cor1}
  Let $f \in R_H^{0}(\gamma, \delta, \lambda)$ ($\lambda=1/2, \gamma=1, \delta= 1$). Then
  $$|z| + \sum_{m=2}^{\infty}(|a_m|+|b_m|)|z|^m + \sum_{m=2}^{\infty}(|a_m|+|b_m|)^2|z|^{2m} \leq d(f(0),\partial f(\mathbb{D}))$$
  for $|z|=r \leq r_2(1,1,{1}/{2})$, where $r_2(1,1,{1}/{2})\simeq0.652442$ is the unique root in $(0,1)$ of the equation
  \begin{equation*}
    -r^2+Li_2(r)+Li_4(r^2)-\dfrac{\pi^2}{12}=0.
  \end{equation*}
  The radius $r_2(1,1,{1}/{2})$ is the best possible.
\end{corollary}

\begin{proof}
  Let $f\in R_H^0(\gamma,\delta,\lambda)$, where $\lambda=1/2,\delta=\gamma=1$ and $|z|=r$. Then by using Lemma \ref{lm1}, we get that
  \begin{align*}
      |z| + \sum_{m=2}^{\infty}(|a_m|+|b_m|)|z|^m + &\sum_{m=2}^{\infty}(|a_m|+|b_m|)^2|z|^{2m}\\
  &\leq r+\sum_{m=2}^{\infty}\dfrac{4(1-\frac{1}{2})r^m}{m^2[2]}+\sum_{m=2}^{\infty}\dfrac{16(1-\frac{1}{2})^2r^{2m}}{m^4[4]}\\
  &=r+\sum_{m=2}^{\infty}\dfrac{2r^m}{2m^2}+\sum_{m=2}^{\infty}\dfrac{4r^{2m}}{4m^4}\\
  &= r+\sum_{m=2}^{\infty}\dfrac{r^m}{m^2}+\sum_{m=2}^{\infty}\dfrac{r^{2m}}{m^4}.
  \end{align*}
  Using the definition of dilogarithm function $Li_2$, it can be easily observed that
  $$ \sum_{m=2}^{\infty}\dfrac{r^m}{m^2}=Li_2(r)-r $$
  $$ \sum_{m=2}^{\infty}\dfrac{r^{2m}}{m^4}=Li_4(r^2)-r^2. $$
  Thus, we get that
  \begin{equation}\label{eqn1}
  |z| + \sum_{m=2}^{\infty}(|a_m|+|b_m|)|z|^m + \sum_{m=2}^{\infty}(|a_m|+|b_m|)^2|z|^{2m} \leq -r^2+Li_2(r)+Li_4(r^2).
  \end{equation}
  A simple calculation shows that
  \begin{eqnarray*}
  -r^2+Li_2(r)+Li_4(r^2) &\leq& 1+ \sum_{m=2}^{\infty}\dfrac{(-1)^{m-1}4(\gamma-\lambda)}{m^2[2\gamma+(\delta-\gamma)(m-1)]}\nonumber\\
  &=& 1+\sum_{m=2}^{\infty}\dfrac{(-1)^{m-1}}{m^2}\nonumber\\
  &=& \dfrac{1}{12}(-12+\pi^2)+1 = \dfrac{\pi^2}{12}
  \end{eqnarray*}
  for $r\leq r_2(1,1,1/2)$, where $r_2(1,1,1/2) \approx 0.652442$ is a root of $K_2(r)=0$ in $(0,1)$ and the function $K_2:[0,1]\rightarrow \mathbb{R}$ is defined by
  $$ K_2(r):=-r^2+Li_2(r)+Li_4(r^2)-\dfrac{\pi^2}{12}.$$
  By following the argument in the proof of Theorem \ref{th2}, it can be easily shown that the function $K_2(r)$ has unique root $r_2(1,1,1/2) \approx 0.652442$ in $(0,1)$. Thus, we get that
  \begin{equation}\label{eqn2}
  -r_2(1,1,{1}/{2})^2+Li_2(r_2(1,1,{1}/{2}))+Li_4(r_2(1,1,{1}/{2})^2)-\dfrac{\pi^2}{12}=0
  \end{equation}
  \begin{center}
\begin{figure}[H]
  \includegraphics[height=5cm, width=8cm]{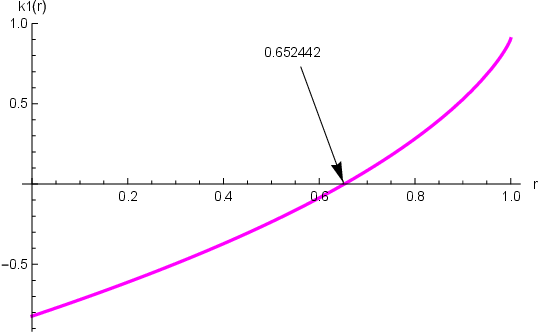}
  \caption{The radius $r_2(1, 1, 1/2) \simeq 0.652442$ is the root of $k_2(r)=0$ }
\end{figure}

\end{center}
  To show that $r_2(1,1,{1}/{2})$ is the best possible, we consider the function $f=f_{2}$ defined by
  \[
  f_2(z):=z+\sum_{m=2}^{\infty}\dfrac{4(1-{1}/{2})z^m}{m^2[2]}=z+\sum_{m=2}^{\infty}\dfrac{z^m}{m^2}.
  \]
  Using equation \eqref{eq16}, it is easy to see that
  \begin{equation*}
    d(f_2(0),\partial f_2(\mathbb{D}))=1+\sum_{m=2}^{\infty}\dfrac{(-1)^{m-1}}{m^2}=\dfrac{\pi^2}{12}.
  \end{equation*}
   By \eqref{eqn1} and \eqref{eqn2}, it follows  that
   \begin{eqnarray*}
   |z| &+& \sum_{m=2}^{\infty}(|a_m|+|b_m|)|z|^m + \sum_{m=2}^{\infty}(|a_m|+|b_m|)^2|z|^{2m}\\
   &=& -r_2(1,1,{1}/{2})^2+Li_2(r_2(1,1,{1}/{2}))+Li_4(r_2(1,1,{1}/{2})^2)\\
   &=& \dfrac{\pi^2}{12} = d(f_2(0),\partial f_2(\mathbb{D})).
   \end{eqnarray*}

   Therefore, $r_2(1,1,{1}/{2})=0.652442$ is the best possible.
\end{proof}

\begin{itemize}
  \item[$\blacksquare$] For $\lambda=1/2, \gamma=1, \delta=1$, in the above Corollary \ref{cor1},  we can see that for different values of $p$ the Bohr radii are $r_3(1,1,\frac{1}{2})\approx 0.659277$, $r_4(1,1,\frac{1}{2})\approx 0.659997$, $r_5(1,1,\frac{1}{2})\approx 0.660074$, $r_6(1,1,\frac{1}{2})\approx 0.660083$, $r_7(1,1,\frac{1}{2})\approx 0.660083$, $r_8(1,1,\frac{1}{2})\approx 0.660084$. In fact $r_p(1,1,\frac{1}{2})\approx 0.660084$.
  \item[$\blacksquare$] For $\lambda=0, \gamma=1, \delta=1$, the Bohr radius $r_2(1,1,0)$ in Theorem \ref{th2} is  the root of the  following equation
      \[
      -r_2(1,1,0)-4r_2(1,1,0)^2+2Li_2(r_2(1,1,0))+4Li_4(r_2(1,1,0)^2)=\frac{\pi^2}{6}-1.
      \]
      An easy computation shows that $r_2(1,1,0)\approx 0.480812$.\\
      For different values of $p$ the Bohr radii are $r_3(1,1,0)\approx 0.487911$, $r_4(1,1,0)\approx 0.488711$, $r_5(1,1,0)\approx 0.488874$, $r_6(1,1,0)\approx 0.488886$, $r_7(1,1,0)\approx 0.488888$, $r_8(1,1,0)\approx 0.488888$. In fact $r_p(1,1,0)\approx 0.488888$ for $p\geq7$.
\end{itemize}
We  observe that for $p=2$ in Theorem \ref{th2}, the  sharp Bohr-Radius $r_p(\gamma, \delta, \lambda)$ depends on the parameters $0\le \lambda < \gamma \le \delta$.
\begin{center}
  \begin{table}[h]
    \centering
     \setlength\tabcolsep{40pt}
    \begin{tabular*}{\linewidth}{@{\extracolsep{\fill}}|c|c|c|c|}
        \toprule
        \textbf{$\lambda$} & \textbf{$\gamma$} & \textbf{$\delta$} & \textbf{$r_2(\gamma, \delta, \lambda)$} \\
        \toprule
        0.125 & 0.1260 & 0.5 & 0.9962 \\
        0.125 & 0.1255 & 0.5 & 0.9981 \\
        0.125 & 0.1254 & 0.5 & 0.9984 \\
        0.125 & 0.1253 & 0.5 & 0.9988 \\
        \bottomrule
    \end{tabular*}
    \caption{The roots $r_2(\gamma, \delta, \lambda)$ for Theorem \ref{th2} as $\gamma$ approaches close to $\lambda$}
    \label{tab:random_numbers}
\end{table}
\end{center}
From the above table 1, we can easily conclude that root is approaching towards 1.0 as $\gamma$ tends to $\lambda$.\\
\begin{figure}
  \centering
  \includegraphics[height=6cm, width=13cm]{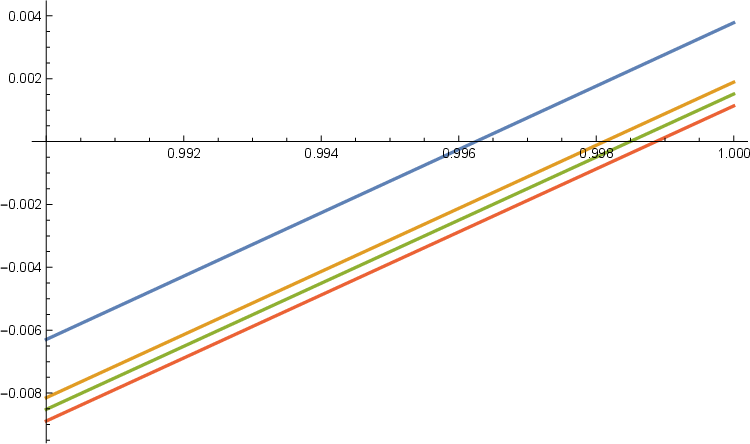}\\
  \caption{The radii $r_p\rightarrow1$  (Theorem \ref{th2}) as $\gamma\rightarrow\lambda=0.125$}
\end{figure}
(Note that plot seems straight line due to very close observation on root axis.)


\begin{theorem}\label{cor2}
  Let $f \in R_H^{0}(\gamma, \delta, \lambda)$. If $\lambda=1/4, \gamma=1/2$ and $\delta= 1$. Then
  $$|z| + \sum_{m=2}^{\infty}(|a_m|+|b_m|)^2|z|^{2m} \leq d(f(0),\partial f(\mathbb{D}))$$
  for $|z|=r \leq r_3$, where $r_3\simeq0.676479$ is the unique root in $(0,1)$ of equation
  \begin{equation*}
    r-2r^2+(1-\frac{2}{r^2})\log(1-r^2)+2Li_2(r^2)-4-\frac{\pi^2}{6}+4\log(2)=0.
  \end{equation*}
  The radius $r_3$ is the best possible.
\end{theorem}

\begin{proof}
Let $f\in R_H^0(\gamma,\delta,\lambda)$. Then for $\delta=1, \gamma=1/2, \lambda=1/4$ and $|z|=r$, using Lemma \ref{lm1}, we have
\begin{eqnarray*}
|z| + \sum_{m=2}^{\infty}(|a_m|+|b_m|)^2|z|^{2m} &\leq& r+\sum_{m=2}^{\infty}\dfrac{4(\frac{1}{2}-\frac{1}{4})r^{2m}}{m^2[1+(1-\frac{1}{2})(m-1)]}\\
&=& r+\sum_{m=2}^{\infty}\dfrac{r^{2m}}{m^2[1+(\frac{1}{2})(m-1)]}\\
&=& r+\sum_{m=2}^{\infty}\dfrac{2 r^{2m}}{m^2[2+(1)(m-1)]}\\
&=& r+\sum_{m=2}^{\infty}\dfrac{2 r^{2m}}{m^2(m+1)}.
\end{eqnarray*}

An easy calculation shows that
\[
\sum_{m=2}^{\infty}\dfrac{2 r^{2m}}{m^2(m+1)}=\sum_{m=2}^{\infty}\dfrac{2 r^{2m}}{m^2}+\sum_{m=2}^{\infty}\dfrac{2 r^{2m}}{m+1}-\sum_{m=2}^{\infty}\dfrac{2 r^{2m}}{m}.
\]
By the definition of dilogarithm function and simple calculations, we have
\[  \left\{
\begin{array}{ll}
\sum_{m=2}^{\infty}\dfrac{2 r^{2m}}{m^2}=-2r^2+2Li_2(r^2)\\ \ \\
\sum_{m=2}^{\infty}\dfrac{2 r^{2m}}{m+1}=\dfrac{2}{r^2}\left(\dfrac{-r^4}{2}-r^2-\log(1-r^2)\right)\\ \ \\
\sum_{m=2}^{\infty}\dfrac{2 r^{2m}}{m}=-r^2-\log(1-r^2).
\end{array}
\right .\]
Thus,
\begin{eqnarray}\label{eq13e1}
\sum_{m=2}^{\infty}\dfrac{2 r^{2m}}{m^2(m+1)} &=& r-2r^2+2Li_2(r^2)+\dfrac{2}{r^2}\left(\dfrac{-r^4}{2}-r^2-\log(1-r^2)\right)+r^2+\log(1-r^2)\\
&=& r-2r^2+2Li_2(r^2)-2+(1-\dfrac{2}{r^2})\log(1-r^2).
\end{eqnarray}
Hence,  equation \eqref{eq16} together with a  simple calculation shows that
$$r-2r^2+2Li_2(r^2)-2+(1-\dfrac{2}{r^2})\log(1-r^2) \leq 2+\dfrac{\pi^2}{6}-4\log(2)$$
for $r\leq r_3$, where $r_3 \approx 0.676479$ is a root of $K_3(r)$ in $(0,1)$ and $K_3:[0,1]\rightarrow \mathbb{R}$ is defined by
\[
K_3(r):=r-2r^2+2Li_2(r^2)-4+(1-\dfrac{2}{r^2})\log(1-r^2)-\dfrac{\pi^2}{6}+4\log(2).
\]
Now using the arguments as done in  Theorem \ref{th2}, we can easily show that $K_3(r)$ has unique root$r_3\approx 0.676479$ in $(0,1)$.
Therefore, $K_3(r_3)=0$ is equivalent to
\begin{equation}\label{eq13e2}
r_3-2r_3^2+2Li_2(r_3^2)-4+(1-\dfrac{2}{r_3^2})\log(1-r_3^2)-\dfrac{\pi^2}{6}+4\log(2)=0.
\end{equation}

\begin{center}
\begin{figure}[H]
  \includegraphics[height=5cm, width=8cm]{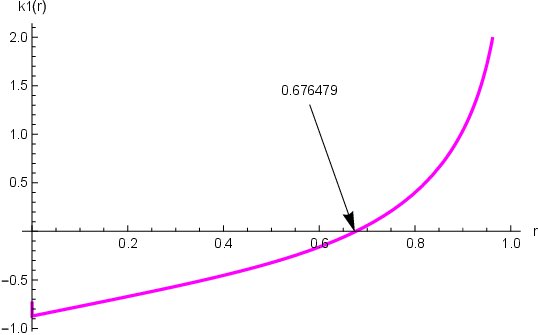}
  \caption{The radius $r_3 \simeq 0.676479$ is the root of $K_3(r)=0$ }
\end{figure}

\end{center}
Next, to show that $r_3$ is best possible, consider the function $f_3$ given by
\[
f_3(z):=z+\sum_{m=2}^{\infty}\dfrac{2z^m}{m^2(m+1)}.
\]
Then a quick computation using equations \eqref{eq13e1} and \eqref{eq13e2} for $f=f_3$ and $r=r_3$ shows that
\begin{eqnarray*}
  |z| + \sum_{m=2}^{\infty}(|a_m|+|b_m|)^2|z|^{2m} &=& r_3+\sum_{m=2}^{\infty}\dfrac{2r_3^m}{m^2(m+1)}\\
  &=& r_3-2r_3^2+2Li_2(r_3^2)+(1-\dfrac{2}{r_3^2})\log(1-r_3^2)\\
  &=& 4+\dfrac{\pi^2}{6}-4\log(2)\\
  &=& d(f_3(0),\partial f_3(\mathbb{D})).
\end{eqnarray*}
Thus, $r_3\approx 0.676479$ is the best possible.
\end{proof}

\begin{theorem}\label{th3}
  If $f \in R_H^{0}(\gamma, \delta, \lambda)$ ($0 \leq \lambda < \gamma \leq \delta$), then
  $$|f(z)| + \sum_{m=2}^{\infty}(|a_m|+|b_m|)|z|^m \leq d(f(0),\partial f(\mathbb{D}))$$
  for $|z|=r \leq R_1(\gamma, \delta, \lambda)$, where $R_1(\gamma, \delta, \lambda)$ is the unique root in $(0,1)$ of
  \begin{equation}\label{eqn3}
    r + 8(\gamma-\lambda)\sum_{m=2}^{\infty} \frac{r^m}{m^2[2\gamma+(\delta - \gamma)(m-1)]}-1 - 4(\gamma-\lambda)\sum_{m=2}^{\infty} \frac{(-1)^{m-1}}{m^2[2\gamma+(\delta - \gamma)(m-1)]}=0.
  \end{equation}
  Here $R_1(\gamma, \delta, \lambda)$ is the best possible.

\end{theorem}

\begin{proof}
  Let $f\in R_H^0(\gamma,\delta,\lambda)$. Then by Lemma \ref{lm1} and Lemma \ref{lm2}, we get that
  \begin{eqnarray}\label{eqn4}
    |f(z)|+\sum_{m=2}^{\infty}(|a_m|+|b_m|)|z|^m &\leq& |z|+4(\gamma-\lambda)\sum_{m=2}^{\infty}\dfrac{|z|^m}{m^2[2\gamma+(\delta-\gamma)(m-1)]}\nonumber\\
    &+& 4(\gamma-\lambda)\sum_{m=2}^{\infty}\dfrac{|z|^m}{m^2[2\gamma+(\delta-\gamma)(m-1)]}\nonumber\\
    &=& r+8(\gamma-\lambda)\sum_{m=2}^{\infty}\dfrac{r^m}{m^2[2\gamma+(\delta-\gamma)(m-1)]}.
  \end{eqnarray}
  Now, quick computations shows that
  \begin{equation*}
    r+\sum_{m=2}^{\infty}\dfrac{8(\gamma-\lambda)r^m}{m^2[2\gamma+(\delta-\gamma)(m-1)]}\leq 1+\sum_{m=2}^{\infty}\dfrac{4(\gamma-\lambda)(-1)^{m-1}}{m^2[2\gamma+(\delta-\gamma)(m-1)]}
  \end{equation*}
  for $r\leq R_1$, where $R_1$ is a root of $K_1(r)$ in $(0,1)$ and $K_1:[0,1]\rightarrow \mathbb{R}$ is defined as
  \begin{equation*}
    K_1(r):=r+\sum_{m=2}^{\infty}\dfrac{8(\gamma-\lambda)r^m}{m^2[2\gamma+(\delta-\gamma)(m-1)]}- 1-\sum_{m=2}^{\infty}\dfrac{4(\gamma-\lambda)(-1)^{m-1}}{m^2[2\gamma+(\delta-\gamma)(m-1)]}.
  \end{equation*}
  Clearly, $K_1(r)$ is real-valued continuous function on $[0,1]$ and differentiable on $(0,1)$.\\
  Since,
  \[
  K_1(0)=- 1-\sum_{m=2}^{\infty}\dfrac{4(\gamma-\lambda)(-1)^{m-1}}{m^2[2\gamma+(\delta-\gamma)(m-1)]} < 0
  \]
  and also observe that
  \begin{eqnarray*}
  K_1(1) &=& \sum_{m=2}^{\infty}\dfrac{8(\gamma-\lambda)}{m^2[2\gamma+(\delta-\gamma)(m-1)]}- \sum_{m=2}^{\infty}\dfrac{4(\gamma-\lambda)(-1)^{m-1}}{m^2[2\gamma+(\delta-\gamma)(m-1)]}\\
  &=& \sum_{m=2}^{\infty}\dfrac{(8-4(-1)^{m-1})(\gamma-\lambda)}{m^2[2\gamma+(\delta-\gamma)(m-1)]} > 0.
  \end{eqnarray*}
  Thus, by Intermediate Value Theorem, $K_1$ has a root in $(0,1)$.
  Now for $r\in (0,1)$,
  \[
  K_1'(r)=1+\sum_{m=2}^{\infty}\dfrac{8m(\gamma-\lambda)r^{m-1}}{m^2[2\gamma+(\delta-\gamma)(m-1)]}>0.
  \]
  Hence, $K_1$ is strictly monotone (increasing) function in $(0,1)$ which shows that $K_1$ has unique root in $(0,1)$.
  Let $r=R_1$ be the unique root of the equation  $K_1(r)=0$, that is,
  \begin{equation}\label{eq27}
    R_1+\sum_{m=2}^{\infty}\dfrac{8(\gamma-\lambda)R_1^m}{m^2[2\gamma+(\delta-\gamma)(m-1)]}= 1+\sum_{m=2}^{\infty}\dfrac{4(\gamma-\lambda)(-1)^{m-1}}{m^2[2\gamma+(\delta-\gamma)(m-1)]}.
  \end{equation}
  To show $R_1$ is the best possible, consider the function $f=F_1$ given by
  \begin{equation}\label{eq28}
    F_1(z):=z+\sum_{m=2}^{\infty}\dfrac{4(\gamma-\lambda)z^m}{m^2[2\gamma+(\delta-\gamma)(m-1)]}.
  \end{equation}
   Using \eqref{eq21}, \eqref{eq27} and \eqref{eq28}, for \( |z|=R_1 \), we get that
  \begin{eqnarray*}
  |F_1(z)|+\sum_{m=2}^{\infty}(|a_m|+|b_m|)|z|^m &=& R_1+\sum_{m=2}^{\infty}\dfrac{8(\gamma-\lambda)R_1^m}{m^2[2\gamma+(\delta-\gamma)(m-1)]}\\
  &=& 1+\sum_{m=2}^{\infty}\dfrac{4(\gamma-\lambda)(-1)^{m-1}}{m^2[2\gamma+(\delta-\gamma)(m-1)]} \\
  &=& d(F_1,\partial F_1(\mathbb{D})).
  \end{eqnarray*}
  Hence $R_1$ is the best possible.
\end{proof}

\begin{corollary}\label{cor3}
Let  $f \in R_H^{0}(\gamma, \delta, \lambda)$. If $\lambda=0, \gamma={1}/{2}$ and $\delta=1$, then
\[
|f(z)| + \sum_{m=2}^{\infty}(|a_m|+|b_m|)|z|^m \leq d(f(0),\partial f(\mathbb{D}))
\]
for $|z|=r\leq R_2$, where $R_2\approx 0.521468$ is the unique root of equation
\begin{equation*}
  -3r+8Li_2(r)+(8-\frac{8}{r})\log(1-r)-11= \dfrac{\pi^2}{3}-8\log(2)
\end{equation*}
in $(0,1)$. The radius $R_2$ is the best possible.\\
\end{corollary}

\begin{proof}
  Let  $f \in R_H^{0}(\gamma, \delta, \lambda)$. For $\lambda=0, \gamma=\dfrac{1}{2}, \delta=1$ and $|z|=r$, it follows from  Theorem \ref{th3} that

  \begin{align}\label{eqn6}
      |f(z)|+\sum_{m=2}^{\infty}(|a_m|+|b_m|)|z|^m &\leq r+\sum_{m=2}^{\infty}\dfrac{8(\frac{1}{2})r^m}{m^2(1+\frac{1}{2}(m-1))}\nonumber\\
    &= r+\sum_{m=2}^{\infty}\dfrac{8r^m}{m^2(m+1)}.
  \end{align}

  A simple calculation yields
  \begin{equation}\label{eqn5}
  \sum_{m=2}^{\infty}\dfrac{8r^m}{m^2(m+1)}=\sum_{m=2}^{\infty}\dfrac{8r^m}{m^2}+ \sum_{m=2}^{\infty}\dfrac{8r^m}{m+1}- \sum_{m=2}^{\infty}\dfrac{8r^m}{m}.
  \end{equation}
  By the definition of dilogarithm function, we have
  \[  \left\{
  \begin{array}{ll}
  \sum_{m=2}^{\infty}\dfrac{8r^{m}}{m^2}&= \,-8r+8Li_2(r)\\
  \sum_{m=2}^{\infty}\dfrac{8r^m}{m+1}&=\,\dfrac{8}{r}\left(\dfrac{-r^2}{2}-r-\log(1-r)\right)\\ \ \\
  \sum_{m=2}^{\infty}\dfrac{8r^m}{m}&=\, -8r-8\log(1-r).
  \end{array}
  \right .\]
From \eqref{eqn6} and \eqref{eqn5}, it follows that
\begin{equation}\label{eq30}
    |f(z)|+\sum_{m=2}^{\infty}(|a_m|+|b_m|)|z|^m \leq -3r+8Li_2(r)-8+(8-\frac{8}{r})\log(1-r).
  \end{equation}
  Also,
  \begin{align}\label{eq16e1}
  1+\sum_{m=2}^{\infty}\dfrac{4(-1)^{m-1}}{m^2(m+1)}
  &=1+4\left(\sum_{m=2}^{\infty}\dfrac{(-1)^m}{m^2}+ \sum_{m=2}^{\infty}\dfrac{(-1)^m}{m+1}- \sum_{m=2}^{\infty}\dfrac{(-1)^m}{m}\right)  \nonumber \\
  &=1+4(\frac{1}{12}(12 -\pi^2)+\frac{1}{2}(-1 + 2 log(2))-1+log(2))\nonumber \\
  &=3+\dfrac{\pi^2}{3}-8\log(2).
  \end{align}
  A simple computation shows that
  \[
  -3r+8Li_2(r)-8+(8-\frac{8}{r})\log(1-r) \leq 3+\dfrac{\pi^2}{3}-8\log(2).
  \]
  for $r\leq R_2$, where $R_2$ is a root of $G_2(r)=0$ in $(0,1)$ and $G_2:[0,1]\rightarrow \mathbb{R}$ is defined by
  \[
  G_2(r):=-3r+8Li_2(r)-8+(8-\frac{8}{r})\log(1-r)-3-\dfrac{\pi^2}{3}+8\log(2).
  \]
  By following as in  Theorem \ref{th3}, it is easy to see that $G_2(r)$ has unique root $R_2\approx 0.521468$.  Thus,
  \begin{equation}\label{eq31}
  -3R_2+8Li_2(R_2)-8+(8-\frac{8}{R_2})\log(1-R_2)=3+\dfrac{\pi^2}{3}-8\log(2).
  \end{equation}
  By \eqref{eq30}, \eqref{eq16e1} and \eqref{eq31}, the result follows.
  \begin{figure}[H]
  \centering
  \includegraphics[height=5cm, width=8cm]{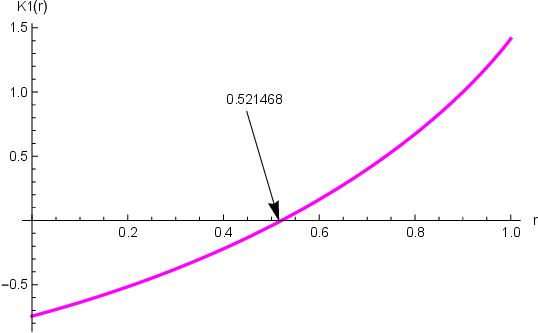}\\
\caption{The radii $r \simeq 0.521468$ vs $K_1(r)$ for corollary \ref{cor3} of Theorem \ref{th3}}
\end{figure}
  The  result is sharp for the function $g=f$ given in Lemma \ref{lm1}. 
\end{proof}

\begin{theorem}\label{th5}

  Let $f\in R_H^{0}(\gamma, \delta, \lambda) (0 \leq \lambda < \gamma \leq \delta)$. If $f=h+\overline{g}$, then
  \begin{enumerate}
  \item $$ |z|+|h(z)|+\sum_{m=2}^{\infty}|a_m||z|^m \leq d(f(0),\partial f(\mathbb{D})) $$
  for $r\leq R_h$, where $R_h$ is the unique root of equation
  \begin{equation*}
    2r+\sum_{m=2}^{\infty} \frac{8(\gamma-\lambda)r^{m}}{m^2[2\gamma+(\delta - \gamma)(m-1)]}-1 - \sum_{m=2}^{\infty} \frac{4(\gamma-\lambda)(-1)^{m-1}}{m^2[2\gamma+(\delta - \gamma)(m-1)]}=0
  \end{equation*}
  in $(0,1)$. Here $R_h$ is the best possible.\\ \ \\
  \item $$ |z|+|g(z)|+\sum_{m=2}^{\infty}|b_m||z|^m \leq d(f(0),\partial f(\mathbb{D})) $$
  for $r\leq R_g$, where $R_g$ is the unique root of equation
  \begin{equation*}
    r+\sum_{m=2}^{\infty} \frac{4(\gamma-\lambda)r^{m}}{m^2[2\gamma+(\delta - \gamma)(m-1)]}-1 - \sum_{m=2}^{\infty} \frac{4(\gamma-\lambda)(-1)^{m-1}}{m^2[2\gamma+(\delta - \gamma)(m-1)]}=0
  \end{equation*}
  in $(0,1)$. Here $R_g$ is the best possible.
  \end{enumerate}
\end{theorem}

\begin{proof}
   \begin{enumerate}
   \item Let $f\in R_H^0(\gamma,\delta,\lambda)$.  Using Lemma \ref{lm1} and Lemma \ref{lm2}  together with \eqref{eq16}, for $|z|=r$, we have
   \begin{equation}\label{eq35}
     |z|+|h(z)|+\sum_{m=2}^{\infty}|a_m||z|^m \leq d(f(0),\partial f(\mathbb{D}))
   \end{equation}
   provided $$2r+\sum_{m=2}^{\infty} \frac{8(\gamma-\lambda)r^{m}}{m^2[2\gamma+(\delta - \gamma)(m-1)]}\leq 1 + \sum_{m=2}^{\infty} \frac{4(\gamma-\lambda)(-1)^{m-1}}{m^2[2\gamma+(\delta - \gamma)(m-1)]}.$$
  Let $K_h:[0,1]\rightarrow \mathbb{R}$ be defined by
  \[
  K_h(r):=2r+\sum_{m=2}^{\infty} \frac{8(\gamma-\lambda)r^{m}}{m^2[2\gamma+(\delta - \gamma)(m-1)]}- 1 - \sum_{m=2}^{\infty} \frac{4(\gamma-\lambda)(-1)^{m-1}}{m^2[2\gamma+(\delta - \gamma)(m-1)]}.
  \]
  Using intermediate value theorem, we can easily show that $K_h$ has unique root in $(0,1)$. Let this unique root be $R_h$, so we get that $K_h(R_h)=0$ and hence,
  \begin{equation}\label{eq36}
    2R_h+\sum_{m=2}^{\infty} \frac{8(\gamma-\lambda)R_h^{m}}{m^2[2\gamma+(\delta - \gamma)(m-1)]} = 1 + \sum_{m=2}^{\infty} \frac{4(\gamma-\lambda)(-1)^{m-1}}{m^2[2\gamma+(\delta - \gamma)(m-1)]}.
  \end{equation}
  Next we need to show that $R_h$ is the best possible. Consider the function $h$ given by equation \eqref{eq3}. Then a simple computation using \eqref{eq22} and \eqref{eq36} for  $r=R_h$ shows that
  \begin{eqnarray*}
    |z|+|h(z)|+\sum_{m=2}^{\infty}|a_m||z|^m &=& 2R_h+\sum_{m=2}^{\infty} \frac{8(\gamma-\lambda)R_h^{m}}{m^2[2\gamma+(\delta - \gamma)(m-1)]}\\
    &=& 1 + \sum_{m=2}^{\infty} \frac{4(\gamma-\lambda)(-1)^{m-1}}{m^2[2\gamma+(\delta - \gamma)(m-1)]}\\
    &=& d(F_h(0),\partial F_h(\mathbb{D})).
  \end{eqnarray*}
  Hence, $R_h$ is the best possible. \\ \ \\
  \item Let $f\in R_H^0(\gamma,\delta,\lambda)$. For $|z|=r$,  using Lemma \ref{lm1} and Lemma \ref{lm2}, the inequality
   \begin{equation}\label{eq37}
     |z|+|g(z)|+\sum_{m=2}^{\infty}|b_m||z|^m \leq d(f(0),\partial f(\mathbb{D}))
   \end{equation}
  holds provided $$r+\sum_{m=2}^{\infty} \frac{4(\gamma-\lambda)r^{m}}{m^2[2\gamma+(\delta - \gamma)(m-1)]}\leq 1 + \sum_{m=2}^{\infty} \frac{4(\gamma-\lambda)(-1)^{m-1}}{m^2[2\gamma+(\delta - \gamma)(m-1)]}.$$
  Let $K_g:[0,1]\rightarrow \mathbb{R}$ be defined by
  \[
  K_g(r):=r+\sum_{m=2}^{\infty} \frac{4(\gamma-\lambda)r^{m}}{m^2[2\gamma+(\delta - \gamma)(m-1)]}- 1 - \sum_{m=2}^{\infty} \frac{4(\gamma-\lambda)(-1)^{m-1}}{m^2[2\gamma+(\delta - \gamma)(m-1)]}.
  \]
  Using intermediate value theorem, it can be  shown that $K_g$ has unique root in $(0,1)$. Let $R_g$ be  the unique root of  $K_g(r)=0$. Hence,
  \begin{equation}\label{eq38}
    R_g+\sum_{m=2}^{\infty} \frac{4(\gamma-\lambda)R_g^{m}}{m^2[2\gamma+(\delta - \gamma)(m-1)]} = 1 + \sum_{m=2}^{\infty} \frac{4(\gamma-\lambda)(-1)^{m-1}}{m^2[2\gamma+(\delta - \gamma)(m-1)]}.
  \end{equation}
  The radius  $R_g$ is the best possible for  the function $g$ given by the equation \eqref{eq5}.  \qedhere
  \end{enumerate}
\end{proof}

\begin{corollary}\label{cor4}
  Let $f \in R_H^{0}(\gamma, \delta, \lambda)$. If $\lambda=0,\gamma=1/2$ and $\delta=1$,  then
  $$ |z|+|g(z)|+\sum_{m=2}^{\infty}|b_m||z|^m \leq d(f(0),\partial f(\mathbb{D}))$$
  for $|z|=r\leq R_g^* \approx 0.594279$, where $R_g^*$ is the unique root of equation
  \begin{equation*}
    -r+4Li_2(r)-4+(4-\frac{4}{r})\log(1-r)=3+\frac{\pi^2}{3}-8\log(2).
  \end{equation*}
  The radius $R_g^*$ is the best possible.
\end{corollary}

\begin{figure}
  \centering
  \includegraphics[height=5cm, width=8cm]{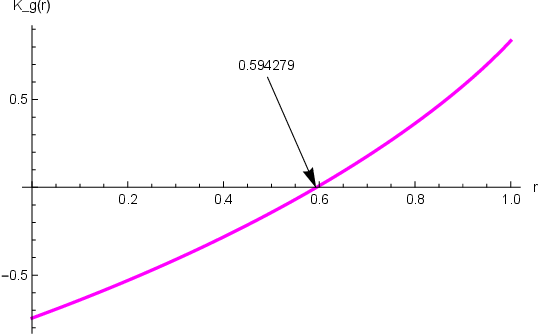}\\
\caption{The radii $R_g^*$  is the roof of the equation $K_g(r)$=0} 
\end{figure}

\begin{proof}
  Let $f\in R_H^0(\gamma,\delta,\lambda)$. For $|z|=r$,  using Lemma \ref{lm1}, we have
   \begin{equation}\label{eq39}
   \begin{aligned}
     |z|+|g(z)|+\sum_{m=2}^{\infty}|b_m||z|^m &\leq r+\sum_{m=2}^{\infty} \frac{4(\frac{1}{2}-0)r^{m}}{m^2[2\frac{1}{2}+\frac{1}{2}(m-1)]}\\
     &= r+\sum_{m=2}^{\infty} \frac{4r^{m}}{m^2(m+1)}.
   \end{aligned}
   \end{equation}
   It  is easy to see that
  \begin{equation}\label{eqn7}
  \sum_{m=2}^{\infty}\dfrac{4r^m}{m^2(m+1)}=\sum_{m=2}^{\infty}\dfrac{4r^m}{m^2}+ \sum_{m=2}^{\infty}\dfrac{4r^m}{m+1}- \sum_{m=2}^{\infty}\dfrac{4r^m}{m}.
  \end{equation}
  By the definition of dilogarithm function, it follows that
  \[  \left\{
  \begin{array}{ll}
  \sum_{m=2}^{\infty}\dfrac{4r^{m}}{m^2}=-4r+4Li_2(r)\\ \ \\
  \sum_{m=2}^{\infty}\dfrac{4r^m}{m+1}=\dfrac{4}{r}\left(\dfrac{-r^2}{2}-r-\log(1-r)\right)\\ \ \\
  \sum_{m=2}^{\infty}\dfrac{4r^m}{m}=-4r-4\log(1-r).
  \end{array}
  \right .\]
  Equations \eqref{eq39} and \eqref{eqn7} yield
\begin{equation*}
    |z|+|g(z)|+\sum_{m=2}^{\infty}|b_m||z|^m \leq -r+4Li_2(r)-4+(4-\frac{4}{r})\log(1-r).
  \end{equation*}
By \eqref{eq16e1}, we have
  \[
  1+\sum_{m=2}^{\infty}\dfrac{4(-1)^{m-1}}{m^2(m+1)}=3+\dfrac{\pi^2}{3}-8\log(2).
  \]
  A simple computation shows that
  \[
  -r+4Li_2(r)-4+(4-\frac{4}{r})\log(1-r) \leq 3+\dfrac{\pi^2}{3}-8\log(2)
  \]
  for $r\leq R_r^*$, where $R_g^*$ is a root of $G_g(r)=0$ in $(0,1)$ and $G_g:[0,1]\rightarrow \mathbb{R}$ is defined by
  \[
  G_2(r):=-r+4Li_2(r)-4+(4-\frac{4}{r})\log(1-r)-3-\dfrac{\pi^2}{3}+8\log(2).
  \]
  By proceeding as in Theorem  \ref{th5}(ii), it is easy to see that $G_g(r)$ has unique root $R_g^*\approx 0.594279$. This shows that $G_g(R_g^*)=0$. Equivalently,
  \begin{equation}\label{eq41}
  -R_g^*+4Li_2(R_g^*)-4+(4-\frac{4}{R_g^*})\log(1-R_g^*)=3+\dfrac{\pi^2}{3}-8\log(2).
  \end{equation}
  The radius  $R_g^*$ is the best possible for  the function
  \[g(z)=z+\sum_{m=2}^{\infty}\frac{2\bar{z^m}}{2\gamma+(m-1)} .\]
The sharpness can be proved by proceeding as in Theorem \ref{th5} .
\end{proof}
\section{Bohr–Rogosinski radius}\label{sec4}
 In this section, we find the  Bohr–Rogosinski radius for the class $R_H^{0}(\gamma, \delta, \lambda)$.

\begin{theorem}\label{th4} Let $n \geq 1$ and $N \geq 2$ be integers.
If $f \in R_H^{0}(\gamma, \delta, \lambda)$ ($0 \leq \lambda < \gamma \leq \delta$), then
  $$|f(z^n)| + \sum_{m=N}^{\infty}(|a_m|+|b_m|)|z|^m \leq d(f(0),\partial f(\mathbb{D}))$$
  for $|z|=r \leq R_{n,N}(\gamma, \delta, \lambda)$, where $R_{n,N}(\gamma, \delta, \lambda)$ is the unique root in $(0,1)$ of
  \begin{equation*}
  \begin{split}
    r^n + \sum_{m=2}^{\infty} \frac{4(\gamma-\lambda)r^{nm}}{m^2[2\gamma+(\delta - \gamma)(m-1)]} &+ \sum_{m=N}^{\infty} \frac{4(\gamma-\lambda)r^{m}}{m^2[2\gamma+(\delta - \gamma)(m-1)]}\\
    &-1 - 4(\gamma-\lambda)\sum_{m=2}^{\infty} \frac{(-1)^{m-1}}{m^2[2\gamma+(\delta - \gamma)(m-1)]}=0.
  \end{split}
  \end{equation*}
  The radius $R_{n,N}(\gamma, \delta, \lambda)$ is the best possible.\\
\end{theorem}

\begin{proof}
  Let $f\in R_H^0(\gamma,\delta,\lambda)$. By  Lemma \ref{lm1} and Lemma \ref{lm2} for $|z|=r$, we have
  \begin{equation}\label{eq32}
  \begin{aligned}
    |f(z^n)| + &\sum_{m=N}^{\infty}(|a_m|+|b_m|)|z|^m\\ &\leq |z|^n + \sum_{m=2}^{\infty}\dfrac{4(\gamma-\lambda)|z|^{mn}}{m^2[2\gamma+(\delta-\gamma)(m-1)]}+ \sum_{m=N}^{\infty}\dfrac{4(\gamma-\lambda)|z|^{m}}{m^2[2\gamma+(\delta-\gamma)(m-1)]}\\
    &= r^n+ \sum_{m=2}^{\infty}\dfrac{4(\gamma-\lambda)r^{mn}}{m^2[2\gamma+(\delta-\gamma)(m-1)]}+ \sum_{m=N}^{\infty}\dfrac{4(\gamma-\lambda)r^{m}}{m^2[2\gamma+(\delta-\gamma)(m-1)]}.
    \end{aligned}
  \end{equation}
  A quick computation shows that
  \begin{align*}
  r^n+ \sum_{m=2}^{\infty}\dfrac{4(\gamma-\lambda)r^{mn}}{m^2[2\gamma+(\delta-\gamma)(m-1)]}+& \sum_{m=N}^{\infty}\dfrac{4(\gamma-\lambda)r^{m}}{m^2[2\gamma+(\delta-\gamma)(m-1)]}\\&\leq 1+\sum_{m=2}^{\infty}\dfrac{4(\gamma-\lambda)(-1)^{m-1}}{m^2[2\gamma+(\delta-\gamma)(m-1)]} \end{align*}
  for $r\leq R_{n,N}(:=R_{n,N}(\gamma, \delta, \lambda))$, where $R_{n,N}$ is the root of $K^*(r)(:=K^*_{n, N, \gamma, \delta, \lambda}(r))=0$ in $(0,1)$, and $K^* : [0,1]\rightarrow \mathbb{R}$ is given by
  \begin{equation*}
  \begin{split}
    K^*(r):=r^n+ \sum_{m=2}^{\infty}\dfrac{4(\gamma-\lambda)r^{mn}}{m^2[2\gamma+(\delta-\gamma)(m-1)]}&+ \sum_{m=N}^{\infty}\dfrac{4(\gamma-\lambda)r^{m}}{m^2[2\gamma+(\delta-\gamma)(m-1)]}\\
    &- 1-\sum_{m=2}^{\infty}\dfrac{4(\gamma-\lambda)(-1)^{m-1}}{m^2[2\gamma+(\delta-\gamma)(m-1)]}.
  \end{split}
     \end{equation*}
  Clearly, $K^*(r)$ is real-valued continuous function on $[0,1]$ and differentiable on $(0,1)$. Also, it can be easily observed that $K^*(0)K^*(1)<0$. The function $K^*(r)$ is strictly monotonic (increasing) for all $r\in (0,1)$ which can be seen from first derivative as
 \[
  (K^*(r))':=nr^{n-1}+ \sum_{m=2}^{\infty}\dfrac{4(\gamma-\lambda)mnr^{mn-1}}{m^2[2\gamma+(\delta-\gamma)(m-1)]}+ \sum_{m=N}^{\infty}\dfrac{4(\gamma-\lambda)mr^{m-1}}{m^2[2\gamma+(\delta-\gamma)(m-1)]}>0.
  \]
  Thus, by intermediate value theorem, $K^*(r)$ has unique root in $(0,1)$, say $R_{n,N}$.\\
  Hence,
  \begin{equation}\label{eq34}
  \begin{split}
    R_{n,N}^n+ \sum_{m=2}^{\infty}\dfrac{4(\gamma-\lambda)R_{n,N}^{mn}}{m^2[2\gamma+(\delta-\gamma)(m-1)]}&+ \sum_{m=N}^{\infty}\dfrac{4(\gamma-\lambda)R_{n,N}^{m}}{m^2[2\gamma+(\delta-\gamma)(m-1)]}\\
    &= 1+\sum_{m=2}^{\infty}\dfrac{4(\gamma-\lambda)(-1)^{m-1}}{m^2[2\gamma+(\delta-\gamma)(m-1)]}.
    \end{split}
  \end{equation}
  Next, to show $R_{n,N}$ is the best possible, we consider the function $f=F^*$ given in Lemma \ref{lm1} defined by
  \[
  F^*(z):= z+\sum_{m=2}^{\infty}\dfrac{4(\gamma-\lambda)z^m}{m^2[2\gamma+(\delta-\gamma)(m-1)]}.
  \]
  By \eqref{eq34} for the function $f=F^*$ and $z=R_{n,N}$, we have
  \begin{eqnarray*}
  |f(z^n)| + \sum_{m=N}^{\infty}(|a_m|+|b_m|)|z|^m &=&  R_{n,N}^n+ \sum_{m=2}^{\infty}\dfrac{4(\gamma-\lambda)R_{n,N}^{mn}}{m^2[2\gamma+(\delta-\gamma)(m-1)]}\\
  &+& \sum_{m=N}^{\infty}\dfrac{4(\gamma-\lambda)R_{n,N}^{m}}{m^2[2\gamma+(\delta-\gamma)(m-1)]}\\
  &=& 1+\sum_{m=2}^{\infty}\dfrac{4(\gamma-\lambda)(-1)^{m-1}}{m^2[2\gamma+(\delta-\gamma)(m-1)]}\\
  &=&\liminf_{r\rightarrow1^-}|F^*(-r)-F^*(0)|\\
  &=& d(F^*(0),\partial F^*(\mathbb{D})).
  \end{eqnarray*}
  This shows that $R_{n,N}$ is the best possible.
\end{proof}

\begin{theorem}\label{th4a} Let $N \geq 2$ be integer.
  If $f \in R_H^{0}(\gamma, \delta, \lambda)$ ( $0 \leq \lambda < \gamma \leq \delta$), then
  $$|f(z)|^2 + \sum_{m=N}^{\infty}(|a_m|+|b_m|)|z|^m \leq d(f(0),\partial f(\mathbb{D}))$$
  for $|z|=r \leq R_{N}(\gamma, \delta, \lambda)$, where $R_{N}(\gamma, \delta, \lambda)$ is the unique root in $(0,1)$ of
  \begin{equation*}
  \begin{split}
    \left(r + \sum_{m=2}^{\infty} \frac{4(\gamma-\lambda)r^{m}}{m^2[2\gamma+(\delta - \gamma)(m-1)]}\right)^2 &+ \sum_{m=N}^{\infty} \frac{4(\gamma-\lambda)r^{m}}{m^2[2\gamma+(\delta - \gamma)(m-1)]}\\
    &-1 - 4(\gamma-\lambda)\sum_{m=2}^{\infty} \frac{(-1)^{m-1}}{m^2[2\gamma+(\delta - \gamma)(m-1)]}=0
  \end{split}
  \end{equation*}
  Here $R_{N}(\gamma, \delta, \lambda)$ is the best possible.\\
\end{theorem}
\begin{proof}
    Let $f\in R_H^0(\gamma,\delta,\lambda)$. By using Lemma \ref{lm1} and Lemma \ref{lm2}, for $|z|=r$, we have
  \begin{equation}\label{eq32a}
  \begin{split}
    |f(z)|^2 &+ \sum_{m=N}^{\infty}(|a_m|+|b_m|)|z|^m\\
    &\leq \left(|z| + \sum_{m=2}^{\infty}\dfrac{4(\gamma-\lambda)|z|^{m}}{m^2[2\gamma+(\delta-\gamma)(m-1)]}\right)^2+ \sum_{m=N}^{\infty}\dfrac{4(\gamma-\lambda)|z|^{m}}{m^2[2\gamma+(\delta-\gamma)(m-1)]}\\
    &= \left(r+ \sum_{m=2}^{\infty}\dfrac{4(\gamma-\lambda)r^{m}}{m^2[2\gamma+(\delta-\gamma)(m-1)]}\right)^2+ \sum_{m=N}^{\infty}\dfrac{4(\gamma-\lambda)r^{m}}{m^2[2\gamma+(\delta-\gamma)(m-1)]}.
    \end{split}
  \end{equation}
  It is easy to see that
   \begin{equation*}
     \begin{split}
       \left(r+ \sum_{m=2}^{\infty}\dfrac{4(\gamma-\lambda)r^{m}}{m^2[2\gamma+(\delta-\gamma)(m-1)]}\right)^2&+ \sum_{m=N}^{\infty}\dfrac{4(\gamma-\lambda)r^{m}}{m^2[2\gamma+(\delta-\gamma)(m-1)]}\\
       &\leq 1+\sum_{m=2}^{\infty}\dfrac{4(\gamma-\lambda)(-1)^{m-1}}{m^2[2\gamma+(\delta-\gamma)(m-1)]}
     \end{split}
   \end{equation*}
  for $r\leq R_{N}$, where $R_{N}(:=R_{N}(\gamma, \delta, \lambda))$ is the root of $K_N(r)(:=K_{N}(\gamma, \delta, \lambda))=0$ in $(0,1)$, and $K_N : [0,1]\rightarrow \mathbb{R}$ is given by
  \begin{equation*}
  \begin{split}
     K_N(r):=\left(r+ \sum_{m=2}^{\infty}\dfrac{4(\gamma-\lambda)r^{m}}{m^2[2\gamma+(\delta-\gamma)(m-1)]}\right)^2&+ \sum_{m=N}^{\infty}\dfrac{4(\gamma-\lambda)r^{m}}{m^2[2\gamma+(\delta-\gamma)(m-1)]}\\
     &- 1-\sum_{m=2}^{\infty}\dfrac{4(\gamma-\lambda)(-1)^{m-1}}{m^2[2\gamma+(\delta-\gamma)(m-1)]}.
  \end{split}
  \end{equation*}
  Clearly, $K_N(r)$ is real-valued continuous function on $[0,1]$ and differentiable on $(0,1)$. Also, it can be easily observed that $K_N(0)K_N(1)<0$. The function $K_N(r)$ is strictly monotonic (increasing) for all $r\in (0,1)$ which can be seen from first derivative as
 \begin{align*}
      (K_N(r))':=& 2 \left (r+ \sum_{m=2}^{\infty}\dfrac{4(\gamma-\lambda)r^{m}}{m^2[2\gamma+(\delta-\gamma)(m-1)]}\right) + 1+ \sum_{m=2}^{\infty}\dfrac{4(\gamma-\lambda)mr^{m-1}}{m^2[2\gamma+(\delta-\gamma)(m-1)]}\\
     &\qquad\qquad\qquad\qquad\qquad\qquad\qquad+ \sum_{m=N}^{\infty}\dfrac{4(\gamma-\lambda)mr^{m-1}}{m^2[2\gamma+(\delta-\gamma)(m-1)]}>0.
  \end{align*}
  Thus, $K_N(r)$ is differentiable and strictly increasing on $(0,1)$, so, by intermediate value theorem $K_N(r)$ has unique root in $(0,1)$, say $R_{N}$.\\
  Therefore,$K_N(R_{N})=0$, which is equivalent to
  \begin{equation}\label{eq34a}
  \begin{split}
    \left(R_{N}+ \sum_{m=2}^{\infty}\dfrac{4(\gamma-\lambda)R_{N}^{m}}{m^2[2\gamma+(\delta-\gamma)(m-1)]}\right)^2&+ \sum_{m=N}^{\infty}\dfrac{4(\gamma-\lambda)R_{N}^{m}}{m^2[2\gamma+(\delta-\gamma)(m-1)]}\\
    &= 1+\sum_{m=2}^{\infty}\dfrac{4(\gamma-\lambda)(-1)^{m-1}}{m^2[2\gamma+(\delta-\gamma)(m-1)]}.
  \end{split}
  \end{equation}
  The result is sharp for  the function $f$ given in Lemma \ref{lm1}.
\end{proof}

\section{Refined Bohr Radius}\label{sec2}
In this section we prove the inequality \eqref{eq13a}, introduced by Ahamed \cite{Aha2024M}, for the class $R_H^{0}(\gamma, \delta, \lambda)$.
Let \( N \) denote a positive integer. We observe that for \( N = 1 \) or \( N = 2 \), \( t \) is calculated as \( \lfloor\frac{(N - 1)}{2}\rfloor = 0 \), resulting in \( \text{sgn}(t) = 0 \). In the instances when \( N = 3 \) or \( N = 4 \), we find that \( t \) equals \( \lfloor \frac{(N - 1)}{2}\rfloor = 1 \). Additionally, for values of \( N \) that are 5 or greater, it follows that \( t \) is at least 2, where \( \text{sgn}(t) = 1 \). In view of the above  observations, we will focus on cases where \( N \) is at least 5 in our next main result.
Next theorem proves the inequality \eqref{eq13a} for $N=1,2,3$ or $4$.
\begin{theorem}\label{th6}
  Let $N\ge 5$ be a positive integer, $t=\lfloor \frac{N-1}{2} \rfloor$, $\mu$, $\beta \in (0,\infty)$ and
  \begin{equation*}
\begin{aligned}
S^f_{\mu, \beta, n, N}(r) := |f(z)|^n + \sum_{m=N}^{\infty} (|a_m| + |b_m|) r^m &+ \mu \, \text{sgn}(t) \sum_{m=1}^{t} (|a_m| + |b_m|)^2 \frac{r^N}{1 - r}\\
&+  \beta \left( 1 + \frac{r}{1 - r} \right) \sum_{m=t+1}^{\infty} (|a_m| + |b_m|)^2 r^{2m}.
\end{aligned}
\end{equation*}

  If $f \in R_H^{0}(\gamma, \delta, \lambda)$ ($0 \leq \lambda < \gamma \leq \delta$) be given by \eqref{eq2} , then $S^f_{\mu, \beta, n, N}(r) \le d(f(0),\partial f(\mathbb{D}))$ for $|z|=r \le R_{\mu, \beta, \gamma, \delta, \lambda}^{n,N,t}$ is the unique root of equation $\psi_{\mu, \beta, \gamma, \delta, \lambda}^{n,N,t}(r)$ in $(0,1)$, where
  \begin{equation*}
    \begin{split}
    \psi_{\mu, \beta, \gamma, \delta, \lambda}^{n,N,t}(r):= (H_{\gamma, \delta,\lambda}(r))^n &+ \sum_{m=N}^{\infty} \frac{4(\gamma-\lambda)r^m}{m^2[2\gamma+(\delta - \gamma)(m-1)]}+\mu F_{t,{\gamma, \delta,\lambda}}^N (r)+ \beta G_{t,{\gamma, \delta,\lambda}}(r)\\
    &-1-\sum_{m=2}^{\infty} \frac{(-1)^{m-1}4(\gamma-\lambda)}{m^2[2\gamma+(\delta - \gamma)(m-1)]}
    \end{split}
  \end{equation*}
  and
  \[  \left\{
  \begin{array}{ll}
  H_{\gamma, \delta,\lambda}(r)=r+\sum_{m=2}^{\infty} \frac{4(\gamma-\lambda)r^m}{m^2[2\gamma+(\delta - \gamma)(m-1)]}\\ \ \\
  F_{t,{\gamma, \delta,\lambda}}^N (r)=sgn(t)\sum_{m=1}^{t} \frac{16(\gamma-\lambda)^2}{m^4[2\gamma+(\delta - \gamma)(m-1)]^2}\dfrac{r^N}{1-r}\\ \ \\
  G_{t,{\gamma, \delta,\lambda}}(r)=(1+\frac{r}{1-r})\sum_{m=t+1}^{\infty} \frac{16(\gamma-\lambda)^2r^{2m}}{m^4[2\gamma+(\delta - \gamma)(m-1)]^2}
  \end{array}
  \right .\]
  The result is sharp.
\end{theorem}
\begin{proof}
  Using Lemma \ref{lm1} and Lemma \ref{lm2} in equation \eqref{eq13a} and performing simple computations, we have
  \begin{equation}\label{eq42}
  \begin{aligned}
  S^f_{\mu, \beta, n, N}(r) &\le (H_{\gamma, \delta,\lambda}(r))^n + \sum_{m=N}^{\infty} \frac{4(\gamma-\lambda)r^m}{m^2[2\gamma+(\delta - \gamma)(m-1)]}+\mu F_{t,{\gamma, \delta,\lambda}}^N (r)+ \beta G_{t,{\gamma, \delta,\lambda}}(r)\\
  &\le 1+\sum_{m=2}^{\infty} \frac{(-1)^{m-1}4(\gamma-\lambda)}{m^2[2\gamma+(\delta - \gamma)(m-1)]}
  \end{aligned}
  \end{equation}
  for $|z|=r \le R_{\mu, \beta, \gamma, \delta,\lambda}^{n,N,t}$, where $R_{\mu, \beta, \gamma, \delta,\lambda}^{n,N,t}$ is the unique root of equation $\psi_{\mu, \beta, \gamma, \delta,\lambda}^{n,N,t}(r)=0$ in (0,1) and $\psi_{\mu, \beta, \gamma, \delta,\lambda}^{n,N,t}:=[0,1] \rightarrow \mathbb{R} $ is defined in statement of Theorem \ref{th6}.\\
  By simple calculation we can easily show that $\dfrac{d}{dr}(\psi_{\mu, \beta, \gamma, \delta,\lambda}^{n,N,t}(r)) > 0$ for $r \in (0,1)$ and hence $\psi_{\mu, \beta, \gamma, \delta,\lambda}^{n,N,t}(r)$ is a real-valued differentiable function on $(0,1)$.
  Also, it can be easily observed that $\psi_{\mu, \beta, \gamma, \delta,\lambda}^{n,N,t}(0)\psi_{\mu, \beta, \gamma, \delta,\lambda}^{n,N,t}(1)<0$. Thus the function $\psi_{\mu, \beta, \gamma, \delta,\lambda}^{n,N,t}(r)$ is strictly increasing on $(0,1)$ so by intermediate value theorem $\psi_{\mu, \beta, \gamma, \delta,\lambda}^{n,N,t}(r)$ has unique root in $(0,1)$, say $R_{\mu, \beta, \gamma, \delta,\lambda}^{n,N,t}$. Therefore $\psi_{\mu, \beta, \gamma, \delta,\lambda}^{n,N,t}(R_{\mu, \beta, \gamma, \delta,\lambda}^{n,N,t})=0$, which is equivalent to
  \begin{equation}\label{eq43}
  \begin{split}
    (H_{\gamma, \delta,\lambda}(R_{\mu, \beta, \gamma, \delta,\lambda}^{n,N,t}))^n + \sum_{m=N}^{\infty} \frac{4(\gamma-\lambda)(R_{\mu, \beta, \gamma, \delta,\lambda}^{n,N,t})^m}{m^2[2\gamma+(\delta - \gamma)(m-1)]}&+\mu F_{t,{\gamma, \delta,\lambda}}^N (R_{\mu, \beta, \gamma, \delta,\lambda}^{n,N,t})+ \beta G_{t,{\gamma, \delta,\lambda}}(R_{\mu, \beta, \gamma, \delta,\lambda}^{n,N,t}) \\
    &= 1+\sum_{m=2}^{\infty} \frac{(-1)^{m-1}4(\gamma-\lambda)}{m^2[2\gamma+(\delta - \gamma)(m-1)]}.
  \end{split}
  \end{equation}
  In view of \eqref{eq16} and \eqref{eq42} it can be seen that
  \begin{equation}\label{eq44}
  S^f_{\mu, \beta, n, N}(r) \le d(f(0),\partial f(\mathbb{D}))
  \end{equation}
  holds.\\
  Next, to show that $R_{\mu, \beta, \gamma, \delta,\lambda}^{n,N,t}$ is the best possible, consider the function $f$ given in equation \eqref{eq3} of Lemma \ref{lm1} defined by
  \[
  f(z) = z + \sum_{m=2}^{\infty} \frac{2(\gamma-\lambda)z^m}{m^2[2\gamma+(\delta - \gamma)(m-1)]}
  \]
  By usual calculations $d(f(0),\partial f(\mathbb{D})) = 1+\sum_{m=2}^{\infty} \frac{(-1)^{m-1}4(\gamma-\lambda)}{m^2[2\gamma+(\delta - \gamma)(m-1)]}$ for above defined function and $|z|=R_{\mu, \beta, \gamma, \delta,\lambda}^{n,N,t}$.\\
  Now, by using equation \eqref{eq16}, \eqref{eq43} and \eqref{eq44} we obtain
  \begin{equation*}
    \begin{aligned}
      S^f_{\mu, \beta, n, N}(r) &= (H_{{\gamma, \delta,\lambda}}(R_{\mu, \beta, \gamma, \delta,\lambda}^{n,N,t}))^n + \sum_{m=N}^{\infty} \frac{4(\gamma-\lambda)(R_{\mu, \beta, \gamma, \delta,\lambda}^{n,N,t})^m}{m^2[2\gamma+(\delta - \gamma)(m-1)]}\\
      &+\mu F_{t,{\gamma, \delta,\lambda}}^N (R_{\mu, \beta, \gamma, \delta,\lambda}^{n,N,t})+ \beta G_{t,{\gamma, \delta,\lambda}}(R_{\mu, \beta, \gamma, \delta,\lambda}^{n,N,t})\\
       &= 1+\sum_{m=2}^{\infty} \frac{(-1)^{m-1}4(\gamma-\lambda)}{m^2[2\gamma+(\delta - \gamma)(m-1)]}\\
       &= d(f(0),\partial f(\mathbb{D})).
    \end{aligned}
  \end{equation*}
  This concludes that $R_{\mu, \beta, \gamma, \delta,\lambda}^{n,N,t}$ is the best possible.
\end{proof}

\begin{theorem}\label{cor5}
  Let $f \in R_H^{0}(\gamma, \delta, \lambda)$ be given by \eqref{eq2} ($0 \leq \lambda < \gamma \leq \delta$) and $\mu, \beta \in (0,\infty)$:
  \begin{enumerate}
    \item If $N=1$, then $S^f_{\mu, \beta, n, 1}(r) \le d(f(0),\partial f(\mathbb{D}))$ for $|z|=r \le R_{\mu, \beta, \gamma, \delta,\lambda}^{n,1,0}$ is the unique root of equation
        \[
        (H_{\gamma, \delta,\lambda}(r))^n+H_{\gamma, \delta,\lambda}(r)-r+\beta G_{0,\gamma, \delta,\lambda}(r)+\dfrac{2(\gamma-\lambda)r}{\gamma}-1-\sum_{m=2}^{\infty} \frac{(-1)^{m-1}4(\gamma-\lambda)}{m^2[2\gamma+(\delta - \gamma)(m-1)]}=0.
        \]
    \item If $N=2$, then $S^f_{\mu, \beta, n, 2}(r) \le d(f(0),\partial f(\mathbb{D}))$ for $|z|=r \le R_{\mu, \beta, \gamma, \delta,\lambda}^{n,2,0}$ is the unique root of equation
        \[
        (H_{\gamma, \delta,\lambda}(r))^n+H_{\gamma, \delta,\lambda}(r)-r+\beta G_{0,\gamma, \delta,\lambda}(r)-1-\sum_{m=2}^{\infty} \frac{(-1)^{m-1}4(\gamma-\lambda)}{m^2[2\gamma+(\delta - \gamma)(m-1)]}=0.
        \]
    \item If $N=3$, then $S^f_{\mu, \beta, n, 3}(r) \le d(f(0),\partial f(\mathbb{D}))$ for $|z|=r \le R_{\mu, \beta, \gamma, \delta,\lambda}^{n,3,1}$ is the unique root of equation
        \begin{align*}
        (H_{\gamma, \delta,\lambda}(r))^n+&H_{\gamma, \delta,\lambda}(r)-r-\dfrac{(\gamma-\lambda)r^2}{2\gamma+(\delta-\gamma)}\\&+\mu \dfrac{8(\gamma-\lambda)^2}{\gamma}\dfrac{r^3}{1-r}+\beta G_{1,\gamma, \delta,\lambda}(r)-1-\sum_{m=2}^{\infty} \frac{(-1)^{m-1}4(\gamma-\lambda)}{m^2[2\gamma+(\delta - \gamma)(m-1)]}=0.
        \end{align*}
    \item If $N=4$, then $S^f_{\mu, \beta, n, 4}(r) \le d(f(0),\partial f(\mathbb{D}))$ for $|z|=r \le R_{\mu, \beta, \gamma, \delta,\lambda}^{n,4,1}$ is the unique root of equation
        \begin{align*}
        (H_{\gamma, \delta,\lambda}(r))^n+&\sum_{m=4}^{\infty} \frac{4(\gamma-\lambda)r^m}{m^2[2\gamma+(\delta - \gamma)(m-1)]}\\+&\mu \dfrac{8(\gamma-\lambda)^2}{\gamma}\dfrac{r^4}{1-r}+\beta G_{1,\gamma, \delta,\lambda}(r)-1-\sum_{m=2}^{\infty} \frac{(-1)^{m-1}4(\gamma-\lambda)}{m^2[2\gamma+(\delta - \gamma)(m-1)]}=0.
        \end{align*}
  \end{enumerate}
  All radii are sharp.
  \end{theorem}
  \begin{proof}
 For $N=1$ and $2$ in part (i) \& (ii), it is easy to see that $t=0$, hence $F_{0,{\gamma, \delta,\lambda}}^1 (r)=F_{0,{\gamma, \delta,\lambda}}^2 (r)=0$. Also, for $N=3$ and $4$ in part (iii) and (iv), $t=1$. Now, proceeding   as in Theorem \ref{th6} we get  the  required result.
\end{proof}

%
%

\section*{Declarations}

\textbf{Conflict of interest}  The author declare that he has no conflict of interest regarding the publication of this paper.

\textbf{Funding}
Author is supported by the J.R.F. with NTA Ref No. 231610192667 of joint CSIR-UGC

\textbf{Data availability statement} Data sharing not applicable to this article as no
datasets were generated or analysed during the current study.
\bibliographystyle{abbrv}
\bibliography{NJ-bibliography}

\end{document}